\documentclass[11pt]{article}
\usepackage[margin=0.86in]{geometry}
\usepackage{amsmath,amssymb,amsthm,mathtools,mathrsfs}
\usepackage{booktabs,longtable,array}
\usepackage{enumitem}
\usepackage{microtype}
\usepackage{aliascnt}
\usepackage[colorlinks=true,linkcolor=blue,citecolor=blue,urlcolor=blue]{hyperref}
\usepackage[nameinlink,capitalise]{cleveref}
\allowdisplaybreaks

\newtheorem{theorem}{Theorem}[section]
\newaliascnt{lemma}{theorem}
\newtheorem{lemma}[lemma]{Lemma}
\aliascntresetthe{lemma}
\newaliascnt{proposition}{theorem}

\aliascntresetthe{proposition}
\newaliascnt{corollary}{theorem}
\newtheorem{corollary}[corollary]{Corollary}
\aliascntresetthe{corollary}
\newaliascnt{conjecture}{theorem}
\newtheorem{conjecture}[conjecture]{Conjecture}
\aliascntresetthe{conjecture}
\aliascntresetthe{corollary}
\newaliascnt{claim}{theorem}
\newtheorem{claim}[claim]{Claim}
\aliascntresetthe{claim}
\theoremstyle{definition}
\newaliascnt{definition}{theorem}

\aliascntresetthe{definition}
\newaliascnt{remark}{theorem}

\aliascntresetthe{remark}

\crefname{theorem}{Theorem}{Theorems}
\crefname{lemma}{Lemma}{Lemmas}
\crefname{proposition}{Proposition}{Propositions}
\crefname{corollary}{Corollary}{Corollaries}
\crefname{claim}{Claim}{Claims}
\crefname{definition}{Definition}{Definitions}
\crefname{remark}{Remark}{Remarks}

\newcommand{\N}{\mathbb N}
\newcommand{\R}{\mathbb R}
\newcommand{\cF}{\mathcal F}
\newcommand{\cC}{\mathcal C}

\newcommand{\cH}{\mathcal H}

\newcommand{\cO}{\mathcal O}
\newcommand{\ex}{\operatorname{ex}}
\newcommand{\circum}{\operatorname{circ}}
\newcommand{\ind}{\mathbf 1}
\newcommand{\fall}[2]{(#1)_{#2}}
\newcommand{\eps}{\varepsilon}
\newcommand{\defeq}{\mathrel{:=}}

\title{Counting Cycles in Graphs with Bounded Circumference}
\author{
Xiamiao Zhao \thanks{\small\raggedright Department of Mathematical Sciences, Tsinghua University, Beijing 100084, China. Email:
\small zxm23@mails.tsinghua.edu.cn}
\qquad
Yuanpei Wang \thanks{\small\raggedright Department of Mathematics, Shanghai University, Shanghai 200444, P.R. China. Email:
\small boyuan@shu.edu.cn}
}
\date{}

\begin{document}
\maketitle

\begin{abstract}
For an integer $L\ge2$, let $a=\lfloor L/2\rfloor$.  Let $H(n,L)$ be the join of $K_a$ and an independent set of order $n-a$, with one extra edge in the independent set when $L$ is odd.  We prove that, for fixed integers $q\ge4$ and $L>q$, and for all sufficiently large $n$, the graph $H(n,L)$ maximizes the number of copies of $C_q$ among all $n$-vertex graphs of circumference at most $L$.  This settles a conjecture of Zhu, Gy\H{o}ri, He, Lv, Salia and Xiao~[Bull. Lond. Math. Soc. 55 (2023)].  For even $q\ge6$, we also prove the boundary case $L=q$.  We further determine the corresponding maximum when a long path is forbidden.
\end{abstract}



\section{Introduction}\label{sec:introduction}

For graphs $T$ and $G$, let $N(T,G)$ be the number of unlabeled copies of $T$ in $G$.  For a family $\cF$ of graphs, define
\[
\ex(n,T,\cF)=\max\{N(T,G):|V(G)|=n\text{ and }G\text{ is }F\text{-free for every }F\in\cF\},
\]
and put $\ex(n,T,F)=\ex(n,T,\{F\})$.  Alon and Shikhelman~\cite{AlonShikhelman2016} studied this problem in a general form.  The classical Tur\'an problem is the case $T=K_2$.

The restrictions considered here come from the classical theory of long paths and cycles.  Erd\H{o}s and Gallai~\cite{ErdosGallai1959} determined the sharp edge bounds for graphs with no long path or no long cycle.  Kopylov~\cite{Kopylov1977} determined the sharp edge bound for two-connected graphs with no long cycle and gave the extremal constructions.  F\"uredi, Kostochka and Verstra\"ete~\cite{FurediKostochkaVerstraete2016} proved that a two-connected graph whose edge count is close to Kopylov's bound either has a long cycle or has a small vertex set whose deletion leaves a star forest.  F\"uredi, Kostochka, Luo and Verstra\"ete~\cite{FurediKostochkaLuoVerstraete2018} completed this stability result by describing the dense two-connected graphs with bounded circumference in terms of Kopylov-type graphs and several exceptional families.  Ma and Ning~\cite{MaNing2020} added a minimum-degree condition and described the corresponding dense two-connected graphs with bounded circumference.  Li and Ning~\cite{LiNing2021} proved that a degree condition on at least half of the vertices guarantees a long path between two fixed vertices.

The first generalized Tur\'an results under a circumference restriction concerned complete graphs.  Luo~\cite{Luo2018} extended Kopylov's theorem from edges to fixed cliques and showed that the Kopylov constructions continue to be extremal.  Lu, Yuan and Zhang~\cite{LuYuanZhang2021} determined the corresponding two-connected problem for complete bipartite graphs.  Lv, Gy\H{o}ri, He, Salia, Xiao and Zhu~\cite{LvEtAl2024} proved the sharp clique bound under an odd circumference restriction.  For forbidden long paths, Gy\H{o}ri, Salia, Tompkins and Zamora~\cite{GyoriSaliaTompkinsZamora2019} obtained asymptotically sharp estimates for several path and cycle targets.

We use the stability theorem of Zhu, Gy\H{o}ri, He, Lv, Salia and Xiao~\cite{ZhuEtAl2023}.  They described two-connected graphs with minimum degree $k$ and circumference at most $2k+1$, and used this result to determine the exact large-order values for $C_4$ and $C_5$.  They proposed the following conjecture.
\begin{conjecture}
\textup{\cite{ZhuEtAl2023}} For fixed integers $q\ge4$ and $L>q$, and all sufficiently large $n$,
\[
\ex\bigl(n,C_q,\cC_{\ge L+1}\bigr)
=N\bigl(C_q,H(n,L)\bigr).
\]
\end{conjecture}
We prove the conjecture for every remaining cycle length.  We also prove the boundary case $L=q$ when $q\ge6$ is even.

For $L\ge2$, put $a=\lfloor L/2\rfloor$ and define
\begin{equation}\label{eq:def-H}
H(n,L)=
\begin{cases}
K_a\vee\overline K_{n-a},&L=2a,\\
K_a\vee\bigl(\overline K_{n-a}+e\bigr),&L=2a+1,
\end{cases}
\end{equation}
where $e$ is one fixed edge in the independent side.  The graph in \eqref{eq:def-H} has circumference at most $L$: a cycle contains at most as many independent-side vertices as clique vertices when $L=2a$, while the extra edge allows at most one additional independent-side vertex when $L=2a+1$.

Our main result is the following.

\begin{theorem}\label{thm:main}
Fix integers $q\ge4$ and $L>q$.  There exists $n_0=n_0(q,L)$ such that, for every $n\ge n_0$,
\[
\ex\bigl(n,C_q,\cC_{\ge L+1}\bigr)=N\bigl(C_q,H(n,L)\bigr).
\]
Moreover, the same equality holds for $q\ge6$ even and $L=q$.
\end{theorem}

The odd boundary case remains open.  We also prove a result for the $(p+1)$-vertex path $P_{p+1}$.

\begin{theorem}\label{thm:path-main}
Fix integers $q\ge6$ and $p\ge1$.
\begin{enumerate}[label=(\roman*),leftmargin=2.4em]
\item If $p\le q-1$, then $\ex(n,C_q,P_{p+1})=0$ for every $n$.
\item If $p=q$, then
\[
\ex(n,C_q,P_{q+1})
=\left\lfloor\frac{n}{q}\right\rfloor\frac{(q-1)!}{2}
\]
for every $n$.
\item If $p\ge q+2$, or if $q$ is even and $p=q+1$, then, for all sufficiently large $n$,
\[
\ex(n,C_q,P_{p+1})=N\bigl(C_q,H(n,p-1)\bigr).
\]
\end{enumerate}
\end{theorem}

\subsection{Proof sketch and organization}

The main new ingredient is a family of fixed-endpoint path estimates.  The stability theorem quoted above applies only after the low-degree vertices have been removed.  To make this reduction sharp, we prove that deleting a vertex of degree at most $a-1$ destroys strictly fewer copies of the target cycle than the corresponding one-vertex increment in $H(n,L)$, at the leading order.  For odd cycles, this requires a finite-kernel reduction and an encoding of the remaining rooted paths by length-two and length-three connectors.  The resulting strict gap is then preserved through the peeling and block-decomposition arguments.

For a two-connected host graph, we repeatedly delete low-degree vertices and bounded terminal blocks.  A telescoping estimate controls the total number of cycles removed by this peeling process.  When the process stops at a large graph of minimum degree at least $a$, the stability theorem above gives the required bound.  We then use the block and component decompositions to pass from two-connected graphs to arbitrary graphs.

\Cref{sec:prelim} contains the counting formulas and the basic path lemmas.  In \cref{sec:mixed} we first treat even cycles, then develop the general reduction and the three fixed-endpoint estimates for odd cycles, and finally prove the low-degree deletion inequalities.  In \cref{sec:terminal} we treat graphs of high minimum degree, carry out the peeling argument, and prove \cref{thm:main}.  Finally, \cref{sec:forbidden-paths} deduces \cref{thm:path-main}.

\section{Preliminaries and notation}\label{sec:prelim}

This section fixes the notation and gives the tools used later.  We first count cycles in the extremal construction and then prove the basic rooted-path and structural lemmas.

All graphs are finite and simple, and all paths and cycles are simple.  The length of a path or cycle is its number of edges.  The symbol $\sqcup$ denotes disjoint union.  We write $V(G)$ and $E(G)$ for the vertex and edge sets, $v(G)=|V(G)|$, $e(G)=|E(G)|$, $\delta(G)$ for the minimum degree, and $\circum(G)$ for the circumference.  We put
\[
\cC_{\ge L}=\{C_\ell:\ell\ge L\}.
\]
For $x\in\R$ and $k\in\N$, define
\[
\fall{x}{k}=x(x-1)\cdots(x-k+1),\qquad \fall{x}{0}=1.
\]
If $x$ is a nonnegative integer and $k>x$, then $\fall{x}{k}=0$.  For a graph $Q$ and an integer $b\ge0$, the notation $bQ$ denotes the disjoint union of $b$ copies of $Q$.  A block is a maximal subgraph with no cutvertex.  As usual, bridges and isolated vertices are allowed as trivial blocks.  A \emph{terminal block} of a graph $G$ is a block containing at most one cutvertex of its connected component.  If it contains one such cutvertex, we denote it by $c(B)$.  If it contains none, then $B$ is a connected component of $G$.  Constants hidden in $O_{\alpha_1,\ldots,\alpha_t}(\cdot)$ may depend only on the displayed fixed parameters.

For a set $X$ and an integer $k\ge0$, put
\[
X_{\ne}^{k}=\{(x_1,\ldots,x_k)\in X^k:x_i\ne x_j\text{ whenever }i\ne j\}.
\]
For distinct $p,q\in V(G)$, let $P_\ell(G;p,q)$ be the number of simple $p$--$q$ paths of length $\ell$.  Thus it is the number of tuples $(v_1,\ldots,v_{\ell-1})\in(V(G)\setminus\{p,q\})_{\ne}^{\ell-1}$ for which $pv_1\cdots v_{\ell-1}q$ is a path.

\subsection{The extremal construction}

Fix $\kappa\in\{0,1\}$ and put $F_{a,s}^{(\kappa)}(n)=N(C_{2s+1},H(n,2a+\kappa))$.

\begin{lemma}\label{lem:candidate-count}
For fixed $a\ge s+1$ and $\kappa\in\{0,1\}$,
\begin{equation}\label{eq:F-asymp}
F_{a,s}^{(\kappa)}(n)=\frac{\fall a{s+1}}2n^s+O_{a,s}(n^{s-1}),
\end{equation}
and
\begin{equation}\label{eq:F-diff}
F_{a,s}^{(\kappa)}(n)-F_{a,s}^{(\kappa)}(n-1)
=s\binom a2\fall{a-2}{s-1}n^{s-1}+O_{a,s}(n^{s-2}).
\end{equation}
\end{lemma}

\begin{proof}
Write $V(H(n,2a+\kappa))=A\sqcup I$, where $|A|=a$ and $|I|=n-a$.  The terms of degree $s$ are the cycles with $s$ vertices in $I$ and $s+1$ vertices in $A$ that do not use the fixed edge in $I$.  Such a cycle has a unique edge with both endpoints in $A$.  Starting from this edge and following the cycle in one of its two directions gives an ordered alternating sequence of the remaining vertices.  Hence the number of these cycles is
\[
\frac{\fall a{s+1}\fall{n-a}s}{2}
=\frac{\fall a{s+1}}2n^s+O_{a,s}(n^{s-1}).
\]
A cycle not counted in this term has at most $s-1$ freely chosen vertices in $I$.  This includes a cycle using the fixed edge in $I$, because the two endpoints of that edge are fixed.  The total number of these remaining cycles is therefore a polynomial in $n$ of degree at most $s-1$.  This proves \eqref{eq:F-asymp}.

The discrete difference of the displayed leading term has leading coefficient $s\fall a{s+1}/2$, while the discrete difference of every remaining term has degree at most $s-2$.  Since $\fall a{s+1}/2=\binom a2\fall{a-2}{s-1}$, we obtain \eqref{eq:F-diff}.
\end{proof}

For even cycles, put $E_{a,s}^{(\kappa)}(n)=N(C_{2s},H(n,2a+\kappa))$.

\begin{lemma}\label{lem:even-candidate-count}
For fixed $a\ge s\ge3$ and $\kappa\in\{0,1\}$,
\begin{align}
E_{a,s}^{(\kappa)}(n)&=\frac{\fall a s}{2s}n^s+O_{a,s}(n^{s-1}),\label{eq:E-asymp}\\
E_{a,s}^{(\kappa)}(n)-E_{a,s}^{(\kappa)}(n-1)
&=\binom a2\fall{a-2}{s-2}n^{s-1}+O_{a,s}(n^{s-2}).\label{eq:E-diff}
\end{align}
\end{lemma}

\begin{proof}
Write $V(H(n,2a+\kappa))=A\sqcup I$, where $|A|=a$ and $|I|=n-a$.  The terms of degree $s$ are the alternating cycles with $s$ vertices in each side.  Choose and cyclically order the clique vertices and then insert the $s$ independent vertices into the $s$ gaps.  Equivalently, there are $\fall a s\fall{n-a}s$ oriented choices, and every unoriented cycle is counted $2s$ times.  Thus their number is
\[
\frac{\fall a s\fall{n-a}s}{2s}
=\frac{\fall a s}{2s}n^s+O_{a,s}(n^{s-1}).
\]
A nonalternating cycle that does not use the fixed edge in $I$ has at most $s-1$ vertices in $I$.  A cycle using that fixed edge has at most $s-2$ other vertices in $I$, so it has at most $s-2$ freely chosen vertices in $I$.  For each possible intersection with $A$ and each cyclic order, the remaining choices are falling factorials in $|I|$.  Hence the number of all nonalternating cycles is a polynomial in $n$ of degree at most $s-1$.  This proves \eqref{eq:E-asymp}.  Its discrete difference has degree at most $s-2$, while the discrete difference of the displayed alternating term has leading coefficient
\[
\frac{\fall a s}{2}
=\binom a2\fall{a-2}{s-2}.
\]
This proves \eqref{eq:E-diff}.
\end{proof}

Set throughout
\begin{equation}\label{eq:h-r-d}
h=s-1,\qquad r=a-2,\qquad d=a-s=r-h+1.
\end{equation}
These parameters are used throughout to shorten the formulas.

\subsection{Weighted and rooted paths}

We use the following weighted path lemma.
\begin{lemma}[Frieze, McDiarmid and Reed~\cite{FMR1992}]\label{lem:FMR}
Let $Q$ be a graph on $m\ge1$ vertices and let $\omega:E(Q)\to\R_{\ge0}$.  Put $\omega(Q)=\sum_{e\in E(Q)}\omega(e)$.  Then $Q$ contains a path $P$ satisfying
\[
\omega(P)\ge\frac{2\omega(Q)}m.
\]
\end{lemma}

\begin{lemma}\label{lem:AB-weight}
Let $P=v_0v_1\cdots v_\ell$ be a path and let $A,B\subseteq V(P)$.  For $1\le i\le\ell$, put
\[
\omega_i=\ind_{\{v_{i-1}\in A,\ v_i\in B\}}+\ind_{\{v_{i-1}\in B,\ v_i\in A\}}.
\]
If
\[
L_{A,B}(P)=\max\{|i-j|:v_i\in A,\ v_j\in B\},
\]
with value $0$ when one of the two sets misses $V(P)$, then
\[
\sum_{i=1}^{\ell}\omega_i\le2L_{A,B}(P).
\]
\end{lemma}

The sets $A$ and $B$ in this lemma need not be disjoint.

\begin{proof}
If $A=\varnothing$ or $B=\varnothing$, the assertion is immediate.  We may therefore assume that all four indices below are defined.
Let $i_A^-$ and $i_A^+$ be the first and last indices occupied by $A$, and define $i_B^-,i_B^+$ similarly.  The number of transitions from $A$ to $B$ is at most
\[
\max\{0,i_B^+-i_A^-\}\le L_{A,B}(P),
\]
and the number of transitions from $B$ to $A$ is at most
\[
\max\{0,i_A^+-i_B^-\}\le L_{A,B}(P).
\]
Adding the two estimates proves the assertion.
\end{proof}

\begin{lemma}\label{lem:rooted-p3}
Let $D\ge2$ be an integer, and let $R$ be an $M$-vertex graph with distinct roots $x,y$.  Suppose that no $x$--$y$ path has length at least $D+1$.  Let
\[
P_3(R;x,y)=|\{(u,w):xuwy\text{ is a path in }R\}|.
\]
Then
\[
P_3(R;x,y)\le(D-2)(M-2)\le(D-2)M.
\]
\end{lemma}

\begin{proof}
If $M=2$, then $P_3(R;x,y)=0$, so the assertion is immediate.  Assume $M\ge3$.
Put
\[
R^\circ=R-\{x,y\},\quad A=N_R(x)\cap V(R^\circ),\quad B=N_R(y)\cap V(R^\circ).
\]
Give an edge $uv\in E(R^\circ)$ the weight
\[
\omega(uv)=\ind_{\{u\in A,\ v\in B\}}+\ind_{\{v\in A,\ u\in B\}}.
\]
Every ordered rooted path $xuwy$ contributes one unit to the weight of its middle edge, and hence
\[
\omega(R^\circ)=P_3(R;x,y).
\]
By \cref{lem:FMR}, $R^\circ$ has a path $P$ with
\[
\omega(P)\ge\frac{2P_3(R;x,y)}{M-2}.
\]
Any subpath of $P$ joining $A$ to $B$ extends, after adding $x$ and $y$, to an $x$--$y$ path in $R$.  Therefore $L_{A,B}(P)\le D-2$.  By \cref{lem:AB-weight},
\[
\omega(P)\le2(D-2).
\]
Combining the last two inequalities proves the lemma.
\end{proof}

\begin{lemma}\label{lem:rooted-shortening}
Let $R$ be a graph with roots $x,y$.  Suppose every edge of $R-\{x,y\}$ is the middle edge of an $x$--$y$ path of length three.  If $R$ contains an $x$--$y$ path of length at least $L\ge3$, then it contains one of length $L$ or $L+1$.
\end{lemma}

\begin{proof}
Let $P=xv_1\cdots v_{\ell-1}y$ be an $x$--$y$ path with $\ell\ge L$.  If $\ell=L$, there is nothing to prove.  Otherwise $v_{L-1}v_L$ is an internal edge of $R$.  Since it is the middle edge of a rooted path of length three, either $v_{L-1}y\in E(R)$ or $v_Ly\in E(R)$.  The corresponding prefix of $P$, followed by this edge, has length $L$ or $L+1$.
\end{proof}

\subsection{Two global structural lemmas}

We use the following two path lemmas.

\begin{lemma}[Menger's theorem, see Diestel~\cite{Diestel2017}]\label{lem:two-connections}
Let $G$ be $2$-connected.
\begin{enumerate}[label=(\roman*),leftmargin=2.4em]
\item If $S_1,S_2\subseteq V(G)$ are disjoint and $|S_1|,|S_2|\ge2$, then $G$ contains two vertex-disjoint $S_1$--$S_2$ paths.
\item If $S_1\cap S_2=\{u\}$ and both $S_1\setminus\{u\}$ and $S_2\setminus\{u\}$ are nonempty, then $G-u$ contains an $(S_1\setminus\{u\})$--$(S_2\setminus\{u\})$ path.
\end{enumerate}
\end{lemma}

\begin{proof}
Part~(i) is the two-set form of Menger's theorem.  Part~(ii) follows from the connectedness of $G-u$.
\end{proof}

\begin{lemma}[Li and Ning~\cite{LiNing2021}]\label{lem:LiNing}
Let $B$ be a $2$-connected graph of order $m$, let $x,y\in V(B)$, and let $k\ge2$.  If at least $(m-1)/2$ vertices of $V(B)\setminus\{x,y\}$ have degree at least $k$ in $B$, then $B$ contains an $x$--$y$ path of length at least $k$.
\end{lemma}

\section{Fixed-endpoint path estimates and low-degree deletion}\label{sec:mixed}

This section proves the path estimates that arise after deleting a low-degree vertex.  We first count the even-length paths needed for even cycles.  For odd cycles, the remaining path has odd length and the argument is more involved.  We first give a general reduction for these paths and then treat three cases: odd circumference with $s\ge4$, even circumference, and odd circumference with $s=3$.  The last subsection converts the path estimates into deletion inequalities.

For $U\subseteq V(J)$ and $\ell\ge1$, let $P_\ell(J;U)$ be the number of unoriented paths of length $\ell$ whose two endpoints lie in $U$.

\subsection{The even-cycle case}

The following estimate is enough for the copies of an even cycle that contain a fixed low-degree vertex.

\begin{lemma}\label{lem:even-rooted-paths}
Fix integers $a\ge2$ and $1\le r\le a-1$, and fix $C>0$.  Let $J$ be an $N$-vertex graph and let $U\subseteq V(J)$ satisfy $|U|\le a-1$ and $e(J)\le CN$.  If $J$ has no path of length $2a$ with both endpoints in $U$, then
\[
P_{2r}(J;U)\le \binom{|U|}{2}\fall{a-2}{r-1}N^r
+O_{a,r,C}\bigl(N^{r-1/(a+r+1)}\bigr).
\]
\end{lemma}

\begin{proof}
For $r=1$, a path is determined by its two endpoints and its middle vertex, so the bound follows.  Assume $r\ge2$.  Form the multihypergraph $\cH=\{N_J(y):y\in V(J)\}$, retaining equal neighborhoods with multiplicity.  It has $N$ hyperedges and total incidence at most $2CN$.  For $S\subseteq V(J)$, put
\[
d_{\cH}(S)=|\{y\in V(J):S\subseteq N_J(y)\}|.
\]

Fix distinct $x,z\in U$ and orient every path from $x$ to $z$.  For $\mathbf y=(y_1,\ldots,y_r)\in V(J)^r$, put $E_i=N_J(y_i)$.  If $x\notin E_1$ or $z\notin E_r$, put $\Psi_{x,z}(\mathbf y)=0$.  Otherwise, for $1\le i\le r-1$, define $A_i=(E_i\cap E_{i+1})\setminus\{x,z\}$, and let $\Psi_{x,z}(\mathbf y)$ be the number of pairwise distinct tuples $(p_1,\ldots,p_{r-1})$ with $p_i\in A_i$.  Every simple path
\[
x y_1p_1y_2p_2\cdots p_{r-1}y_rz
\]
is counted once, whereas the sum may also count nonsimple sequences.  Hence
\begin{equation}\label{eq:even-path-sum}
P_{2r}(J;x,z)\le\sum_{\mathbf y\in V(J)^r}\Psi_{x,z}(\mathbf y).
\end{equation}

First suppose that $d_J(y)\le D$ for every $y$, where $D$ will be chosen later.  Call $\mathbf y$ bad if $\Psi_{x,z}(\mathbf y)>\fall{a-2}{r-1}$.  We claim that there are $O_{a,r}(D^{a+1}N^{r-1})$ bad tuples.  Put $R=\bigcup_{i=1}^{r-1}A_i$.  A bad tuple has $|R|\ge a-1$.  Otherwise it admits at most $\fall{a-2}{r-1}$ choices.  Since it has a valid choice of the $p_i$, choose $a-1$ distinct vertices of $R$ meeting every $A_i$, order them as $w_1,\ldots,w_{a-1}$ in nondecreasing order of an assigned layer, and put $w_0=x$ and $w_a=z$.

Assign the step $w_{j-1}w_j$ to position $i$ if its endpoints lie in $E_i$: use position $1$ for the first step, position $r$ for the last step, position $i$ inside layer $i$, and position $i+1$ between layers $i$ and $i+1$.  Every layer occurs, so no step skips a layer.  For each used position $i$, let $R_i$ be the set of endpoints of its assigned steps.  Then $|R_i|\le a+1$ and
\[
R_i\subseteq\{x,z\}\cup E_{i-1}\cup E_{i+1},
\]
where a nonexistent term is omitted.  If $d_{\cH}(R_i)\ge2a+1$ at every used position, choose a distinct common neighbor for each of the $a$ steps.  At one step at most the $a+1$ vertices $w_0,\ldots,w_a$ and the at most $a-1$ previously chosen representatives are forbidden, so one of the $2a+1$ candidates remains.  The chosen vertices form a simple $x$--$z$ path of length $2a$, a contradiction.  Thus some used position $i$ satisfies $d_{\cH}(R_i)\le2a$.

Fix this position and all entries of $\mathbf y$ except $y_i$.  The displayed containment and the degree bound show that $R_i$ is a subset of a fixed set of order at most $2D+2$ and has order at most $a+1$.  Hence it has $O_a(D^{a+1})$ possible values.  Each value permits at most $2a$ choices for $y_i$.  Summing over $i$ and the other $r-1$ entries proves the claim.  A bad tuple contributes at most $D^{r-1}$ to \eqref{eq:even-path-sum}, and every other tuple contributes at most $\fall{a-2}{r-1}$.  Therefore
\begin{equation}\label{eq:even-bounded-degree-path}
P_{2r}(J;x,z)\le\fall{a-2}{r-1}N^r+O_{a,r}(D^{a+r}N^{r-1}).
\end{equation}

We remove the degree bound.  Let $Y_{>D}=\{y:d_J(y)>D\}$.  Since $e(J)\le CN$, we have $|Y_{>D}|\le2CN/D$.  For a fixed hyperedge $H\in\cH$ and position $j$,
\begin{equation}\label{eq:even-chain-sum}
\sum_{\substack{E_1,\ldots,E_r\in\cH\\E_j=H}}
\prod_{i=1}^{r-1}|E_i\cap E_{i+1}|\le (2e(J))^{r-1}.
\end{equation}
Indeed, for every $E\in\cH$, $\sum_{F\in\cH}|E\cap F|=\sum_{v\in E}d_J(v)\le2e(J)$.  Summing successively to the left and right proves \eqref{eq:even-chain-sum}.  Since $\Psi_{x,z}(\mathbf y)\le\prod_i|E_i\cap E_{i+1}|$, all tuples with a high-degree entry contribute $O_{r,C}(N^r/D)$.  The other tuples satisfy the proof of \eqref{eq:even-bounded-degree-path}.  Thus
\[
P_{2r}(J;x,z)\le\fall{a-2}{r-1}N^r
+O_{a,r,C}(D^{a+r}N^{r-1}+N^r/D).
\]
Choose $D=\lfloor N^{1/(a+r+1)}\rfloor$ and sum over the $\binom{|U|}{2}$ unordered endpoint pairs.
\end{proof}

We now turn to odd cycles.  We first isolate the reduction used in all three cases.

\subsection{A general reduction for odd cycles}\label{sec:reduction}

The next three lemmas reduce the paths that contribute a term of order $N^h$ to two explicitly defined forms and give the selection and encoding tools used later.

Let $K\subseteq V(J)$ contain $p$ and $q$, and put $\cO=V(J)\setminus K$.  For $y\in\cO$, define its type by $T(y)=N_J(y)\cap K$.  For $S\subseteq K$, put $\cO_S=\{y\in\cO:T(y)=S\}$ and $\mathscr S=\{S\subseteq K:|\cO_S|\ge M\}$.  A family of length-three $c$--$c'$ paths is called $M$-disjoint if it contains $M$ paths whose internal vertex pairs are pairwise disjoint.

\begin{lemma}\label{lem:kernel-normal-form}
Fix $a\ge h+2$, $h\ge2$, $\kappa\in\{0,1\}$, $\eps>0$, and an integer $M\ge1$.  There are constants $N_0$ and $C$, depending only on these parameters, such that the following holds.  Let $N\ge N_0$, let $J$ be an $N$-vertex graph with $\circum(J)\le2a+\kappa$, and let $p\ne q$ be vertices of $J$.  Then there are a set $K\subseteq V(J)$ with $p,q\in K$ and $|K|=O_{a,h,\eps,M}(1)$, and a set
\[
\mathscr P\subseteq\{(c,c')\in K^2:c\ne c'\}.
\]
For every $(c,c')\in\mathscr P$, there is an $M$-disjoint family $\mathsf R(c,c')$ of ordered paths $(c,u,v,c')$ of length three in $J$ with $u,v\in V(J)\setminus K$.  For every ordered pair of distinct vertices $(c,c')\in K^2\setminus\mathscr P$, put $\mathsf R(c,c')=\varnothing$.

Apart from at most $\eps N^h+CN^{h-1}$ paths, every $p$--$q$ path of length $2h+1$ has one of the following forms.
\begin{enumerate}[label=(D),leftmargin=3.1em]
\item There are distinct vertices $c_0=p,c_1,\ldots,c_{h+1}=q$ in $K$ and an index $\ell\in\{0,\ldots,h\}$ such that $c_\ell c_{\ell+1}\in E(J[K])$.  The ordered path is obtained from $c_0,c_1,\ldots,c_{h+1}$ by replacing every pair $c_i,c_{i+1}$ with $i\ne\ell$ by $c_i,w_i,c_{i+1}$, where $w_i\in\cO$, $T(w_i)\in\mathscr S$, and $\{c_i,c_{i+1}\}\subseteq T(w_i)$.
\item[(R)] There are distinct vertices $x_0=p,x_1,\ldots,x_h=q$ in $K$ and an index $j\in\{0,\ldots,h-1\}$.  In the ordered path, the part from $x_j$ to $x_{j+1}$ is a member of $\mathsf R(x_j,x_{j+1})$.  Every other pair $x_i,x_{i+1}$ is replaced by $x_i,w_i,x_{i+1}$, where $T(w_i)\in\mathscr S$ and $\{x_i,x_{i+1}\}\subseteq T(w_i)$.
\end{enumerate}
All internal vertices displayed in either form are distinct.
\end{lemma}

\begin{proof}
Every subgraph of $J$ has no cycle of length at least $2a+2$.  The Erd\H{o}s--Gallai cycle theorem~\cite{ErdosGallai1959} therefore gives $e(H)=O_a(|H|)$ for every subgraph $H\subseteq J$.  In particular, $J$ has an orientation with maximum outdegree at most $D=2a+1$: repeatedly remove a vertex of degree at most $D$ and orient its incident edges towards the vertices removed later.

Fix this orientation.  On an oriented copy of a path, call an internal vertex a source if both incident path edges are directed away from it.  Internal sources are pairwise nonadjacent, so a path of length $2h+1$ has at most $h$ internal sources.  Once the internal sources and the fixed endpoints $p,q$ are given, every other path vertex can be reached by following directed edges from one of them.  Since every outdegree is at most $D$, the number of paths with at most $h-1$ internal sources is $O_{a,h}(N^{h-1})$.

We next consider an oriented path with exactly $h$ internal sources.  These sources cover $2h$ different path edges.  Hence there is one path edge which is not incident with an internal source.  Choose $\rho>0$, to be determined later, and put
\[
K_0=\{p,q\}\cup\{z:d^-(z)\ge\rho N\}.
\]
Since $e(J)=O_a(N)$, the order of $K_0$ is bounded in terms of $a$ and $\rho$.  Every vertex outside $K_0$ has indegree less than $\rho N$.

Every internal vertex which is neither a source nor incident with the uncovered edge is the common head of its two adjacent sources.  The number of choices in which such a common head lies outside $K_0$ is bounded by
\[
\sum_{z\notin K_0}d^-(z)^2
\le \rho N\sum_zd^-(z)=O_a(\rho N^2).
\]
After the other $h-2$ sources have been chosen, all remaining vertices have only $O_{a,h}(1)$ choices.  Summing over the possible positions gives $O_{a,h}(\rho N^h)$ paths.

We may therefore assume that every such internal non-source vertex belongs to $K_0$.  Paths with an internal source in $K_0$ contribute only $O_{a,h,\rho}(N^{h-1})$: after its position and its vertex are fixed, at most $h-1$ sources remain to be chosen.  We discard these paths as well.

If both ends $u,v$ of the uncovered edge lie outside $K_0$, orient that edge as $u\to v$.  The two adjacent sources must send an edge to $u$ and to $v$, respectively.  The number of choices for these four vertices is at most
\[
\sum_{\substack{u\to v\\u,v\notin K_0}}d^-(u)d^-(v)
\le \rho N\sum_u d^+(u)d^-(u)
\le D\rho N\sum_ud^-(u)=O_a(\rho N^2).
\]
This again gives $O_{a,h}(\rho N^h)$ paths.  Choose $\rho$ so that the sum of the two terms of this order is at most $\eps N^h$.  Outside these errors, the uncovered edge has both ends in $K_0$ or exactly one end outside $K_0$.  In the first case the path has the preliminary form~\textup{(D)} with kernel $K_0$.  In the second case the outside endpoint of the uncovered edge is followed or preceded by an internal source, and these two vertices form a path of length three between two vertices of $K_0$.  Every other source gives a path of length two between two vertices of $K_0$.

We now select the length-three connector families without changing types repeatedly.  For distinct $c,c'\in K_0$, let $\mathsf R_0(c,c')$ be the set of paths $(c,u,v,c')$ with $u,v\notin K_0$ such that either $u\to c$ and $u\to v$, or $v\to u$ and $v\to c'$.  Thus one of $u,v$ is the internal source of the path.  Let $\nu(c,c')$ be the largest number of members of $\mathsf R_0(c,c')$ with pairwise disjoint internal vertex pairs.

List all ordered pairs of distinct vertices of $K_0$ as $(c_1,c'_1),\ldots,(c_q,c'_q)$ so that $\nu_i=\nu(c_i,c'_i)$ and $\nu_1\le\cdots\le\nu_q$, where $q=|K_0|(|K_0|-1)$.  Put $S_0=0$.  Starting with $k=0$, stop if $k=q$ or $\nu_{k+1}\ge M+2S_k$.  Otherwise replace $k$ by $k+1$, put $S_k=\sum_{i=1}^k\nu_i$, and repeat.  Thus $k$ is the number of initial families selected, and $S_k$ is the sum of their matching numbers.  At every increase, $S_{k+1}<3S_k+M$, so the final value of $S_k$ is bounded in terms of $M$ and $q$.

For every $1\le i\le k$, choose a maximum matching $\mathcal M_i$ in $\mathsf R_0(c_i,c'_i)$, and let $Z$ be the union of the internal vertices of all paths in $\bigcup_{i=1}^k\mathcal M_i$.  Each path has two internal vertices, so $|Z|\le2\sum_{i=1}^k\nu_i=2S_k$.  Since each $\mathcal M_i$ is maximal, every member of $\mathsf R_0(c_i,c'_i)$ with $i\le k$ has an internal vertex in $Z$.

Put $K=K_0\cup Z$.  If $k<q$, then, for $k<i\le q$, let $\mathsf R(c_i,c'_i)$ consist of the members of $\mathsf R_0(c_i,c'_i)$ whose internal vertices avoid $Z$.  A maximum matching in $\mathsf R_0(c_i,c'_i)$ loses at most $|Z|$ members, because its paths have pairwise disjoint internal vertex pairs.  By the stopping rule, the ordering of the $\nu_i$, and $|Z|\le2S_k$, the remaining matching has size at least
\[
\nu_i-|Z|\ge\nu_{k+1}-2S_k\ge M.
\]
The ordered pairs $(c_i,c'_i)$ with $i>k$ form $\mathscr P$.  If $k=q$, put $\mathscr P=\varnothing$.  In either case, put $\mathsf R(c,c')=\varnothing$ for every ordered pair in $K_{\ne}^2\setminus\mathscr P$.

Paths having an internal source in $Z$ contribute $O_{a,h,\rho,M}(N^{h-1})$ and are discarded.  Consider a preliminary path of form~\textup{(R)} whose length-three part meets $Z$.  Its internal vertex in $Z$ is then the non-source vertex.  After that vertex is added to the kernel, the same part consists of one kernel edge and one length-two path, so the whole path has form~\textup{(D)} with kernel $K$.  If the length-three part avoids $Z$, its family cannot be among the first $k$ families, and it belongs to the corresponding $\mathsf R(c,c')$.  Thus all remaining paths have form~\textup{(D)} or~\textup{(R)} relative to $K$.

It remains only to define the final types.  Put $\cO=V(J)\setminus K$, $T(y)=N_J(y)\cap K$, $\cO_S=\{y\in\cO:T(y)=S\}$, and $\mathscr S=\{S:|\cO_S|\ge M\}$.  The union of the classes not in $\mathscr S$ has fewer than $M2^{|K|}$ vertices.  If a source used by a length-two part lies in this union, fixing that source leaves at most $h-1$ sources to be chosen.  All such paths therefore contribute $O_{a,h,\rho,M}(N^{h-1})$.  After discarding them, every length-two part has a type in $\mathscr S$.  The set $K$ and all constants above are bounded in terms of $a,h,\eps,M$.  This proves the lemma.
\end{proof}

For distinct $c,c'\in S\in\mathscr S$, define the set of length-two connectors
\[
\mathsf B_S(c,c')=\{(c,y,c'):y\in\cO_S\}.
\]

\begin{lemma}\label{lem:connector-selection}
Assume the notation and conclusions of \cref{lem:kernel-normal-form}.  Let $Z\subseteq V(J)$, and let $\mathcal Q_1,\ldots,\mathcal Q_t$ be connector families.  Each $\mathcal Q_i$ is either $\mathsf B_S(c,c')$ for some $S\in\mathscr S$ and distinct $c,c'\in S$, or $\mathsf R(c,c')$ for some $(c,c')\in\mathscr P$.  If $t+|Z|\le M/4$, then one can choose $Q_i\in\mathcal Q_i$ for every $i$ so that the internal vertices of $Q_1,\ldots,Q_t$ are distinct and avoid $Z$.
\end{lemma}

\begin{proof}
Choose the length-three connectors first.  For each such family, use the $M$ members with pairwise disjoint internal vertex pairs.  Before any one choice, fewer than $M/2$ vertices are forbidden.  Hence fewer than $M/2$ of these paths meet the forbidden set internally, so an available path remains.  Add its two internal vertices to the forbidden set and continue.  After all length-three connectors have been chosen, choose the internal vertex of each length-two connector from its set $\cO_S$.  At every step fewer than $M$ vertices are forbidden, whereas $|\cO_S|\ge M$.  The greedy choices give the required paths.
\end{proof}

We now give the exact encoding used in the counting argument.  Put $K^\circ=K\setminus\{p,q\}$ and set $x_0=p$, $x_h=q$.  For $j\in\{0,\ldots,h-1\}$ and $t\in\{1,\ldots,h-1\}$, put
\[
(\alpha_j(t),\beta_j(t))=
\begin{cases}
(t-1,t),&t\le j,\\
(t,t+1),&t>j.
\end{cases}
\]
For $\mathbf y=(y_1,\ldots,y_{h-1})\in\cO^{h-1}$, set $T_t=T(y_t)$.  If some $T_t\notin\mathscr S$, put $C_j(\mathbf y)=\varnothing$.  Otherwise let
\[
C_j(\mathbf y)=\left\{\mathbf x\in(K^\circ)_{\ne}^{h-1}:
\{x_{\alpha_j(t)},x_{\beta_j(t)}\}\subseteq T_t
\text{ for }1\le t\le h-1\right\}.
\]
For $\mathbf x\in C_j(\mathbf y)$, put $r_0=x_j$, $r_1=x_{j+1}$, and define
\[
E_j(\mathbf y,\mathbf x)=\{(u,v)\in\cO^2:(r_0,u,v,r_1)\in\mathsf R(r_0,r_1)\}.
\]
Also define
\begin{align}
L_j^+(\mathbf y,\mathbf x)
={}&\left\{(z,w)\in K^\circ\times\cO:
\begin{array}{l}
z\notin\{x_1,\ldots,x_{h-1}\},\quad r_0z\in E(J[K]),\\
\{z,r_1\}\subseteq T(w),\quad T(w)\in\mathscr S
\end{array}
\right\}, \label{eq:Lj-plus}\\
L_{h-1}^-(\mathbf y,\mathbf x)
={}&\left\{(w,z)\in\cO\times K^\circ:
\begin{array}{l}
z\notin\{x_1,\ldots,x_{h-1}\},\quad \{r_0,z\}\subseteq T(w),\\
zr_1\in E(J[K]),\quad T(w)\in\mathscr S
\end{array}
\right\}, \label{eq:Lj-minus}
\end{align}
where $L_j^-=\varnothing$ for $j<h-1$, and put $L_j=L_j^+\cup L_j^-$.

\begin{lemma}\label{lem:canonical-reduction}
Under the hypotheses and notation of \cref{lem:kernel-normal-form},
\begin{equation}\label{eq:canonical-majorization}
P_{2h+1}(J;p,q)
\le
\sum_{j=0}^{h-1}\sum_{\mathbf y\in\cO^{h-1}}
\sum_{\mathbf x\in C_j(\mathbf y)}
\bigl(|L_j(\mathbf y,\mathbf x)|+|E_j(\mathbf y,\mathbf x)|\bigr)
+\eps N^h+CN^{h-1}.
\end{equation}
\end{lemma}

\begin{proof}
Let $\mathcal P_R$ and $\mathcal P_D$ be the sets of paths of forms~\textup{(R)} and~\textup{(D)}, respectively, that remain after the exceptional paths in \cref{lem:kernel-normal-form} are removed.  Consider first a path in $\mathcal P_R$.  Its $h-1$ length-two connectors, read from $p$ to $q$, give $\mathbf y$, and its internal kernel vertices give $\mathbf x$.  The definitions give $\mathbf x\in C_j(\mathbf y)$ and $(u,v)\in E_j(\mathbf y,\mathbf x)$.  This record reconstructs the path, and hence
\[
|\mathcal P_R|
\le\sum_{j=0}^{h-1}\sum_{\mathbf y\in\cO^{h-1}}
\sum_{\mathbf x\in C_j(\mathbf y)}|E_j(\mathbf y,\mathbf x)|.
\]

Now consider a path in $\mathcal P_D$, and use the notation $c_0,\ldots,c_{h+1}$ from that lemma.  If the kernel edge is $c_\ell c_{\ell+1}$ with $\ell<h$, combine it with the following length-two connector $c_{\ell+1}wc_{\ell+2}$.  Delete $c_{\ell+1}$ from the internal kernel tuple and delete $w$ from the list of connector vertices.  The resulting element belongs to $L_\ell^+$.  If the kernel edge is the final edge $c_hq$, combine it with the preceding connector $c_{h-1}wc_h$.  The resulting element belongs to $L_{h-1}^-$.  Simplicity of the original path gives all distinctness conditions in the definitions.  The record consisting of $j$, $\mathbf y$, $\mathbf x$, and the chosen element of $L_j^+$ or $L_{h-1}^-$ uniquely reconstructs the ordered vertex sequence.  Therefore
\[
|\mathcal P_D|
\le\sum_{j=0}^{h-1}\sum_{\mathbf y\in\cO^{h-1}}
\sum_{\mathbf x\in C_j(\mathbf y)}|L_j(\mathbf y,\mathbf x)|.
\]
The paths outside $\mathcal P_R\cup\mathcal P_D$ number at most $\eps N^h+CN^{h-1}$.  Here $\eps N^h$ is the error chosen in the construction of $K$, and $CN^{h-1}$ collects all families for which at most $h-1$ source vertices remain free.  Adding the last two displays and these two error terms proves \eqref{eq:canonical-majorization}.
\end{proof}

\subsection{Odd cycles with an odd circumference bound}

Throughout this subsection $\kappa=1$, $J$ has no $p$--$q$ path of length at least $2a$, and
\[
h=s-1\ge3,\qquad r=a-2,\qquad d=a-s.
\]
We use the data from \cref{lem:kernel-normal-form} with $M=20(a+h+1)$.
Fix $\mathbf y=(y_1,\ldots,y_{h-1})\in\cO^{h-1}$ and put
\[
T_t=T(y_t)\quad(1\le t\le h-1),
\qquad
B=B(\mathbf y)=\left(\bigcup_{t=1}^{h-1}T_t\right)\setminus\{p,q\},
\qquad b=|B|.
\]
If some $T_t\notin\mathscr S$, then $C_j(\mathbf y)=\varnothing$.  We therefore assume $T_t\in\mathscr S$ for every $t$.

For $j\in\{0,\ldots,h-1\}$, put
\[
\lambda_j=
\begin{cases}
1,&0\le j<h-1,\\
2,&j=h-1.
\end{cases}
\qquad
\sum_{j=0}^{h-1}\lambda_j=h+1.
\]
The value $2$ at the final position records the two possibilities in \eqref{eq:Lj-plus} and \eqref{eq:Lj-minus}.

\subsubsection{Insertion into length-two connectors}

For $\mathbf x\in C_j(\mathbf y)$, the $t$-th connector has anchors
\begin{equation}\label{eq:block-anchors}
(\ell_t^{(j)},r_t^{(j)})=
\begin{cases}
(x_{t-1},x_t),&t\le j,\\
(x_t,x_{t+1}),&t>j.
\end{cases}
\end{equation}
Both anchors belong to $T_t$.

\begin{lemma}\label{lem:connector-insertion}
Fix $j$, $\mathbf y$, and $\mathbf x\in C_j(\mathbf y)$.  Let
\[
S\subseteq B\setminus\{x_1,\ldots,x_{h-1}\}.
\]
Every vertex of $S$ can be inserted exactly once into the $T_t$-connectors.  If the number of requested connectors and forbidden vertices is at most $M/4$, their internal vertices can be chosen to be distinct and to avoid the forbidden set.  The total length of these connectors after the insertion is
\[
2\bigl(h-1+|S|\bigr).
\]
In particular, inserting every vertex of $B\setminus\{x_1,\ldots,x_{h-1}\}$ gives total length $2b$.
\end{lemma}

\begin{proof}
For each $z\in S$, choose an index $t(z)$ with $z\in T_{t(z)}$, and put
\[
S_t=\{z\in S:t(z)=t\}.
\]
Order $S_t$ as $z_{t,1},\ldots,z_{t,m_t}$.  Replace the $t$-th anchor pair in \eqref{eq:block-anchors} by the kernel sequence
\[
\ell_t^{(j)},z_{t,1},\ldots,z_{t,m_t},r_t^{(j)}.
\]
Every vertex in this sequence belongs to $T_t$.  Hence consecutive kernel vertices can be joined by members of the corresponding sets $\mathsf B_{T_t}$.  The required internal vertices are distinct by \cref{lem:connector-selection}.  There are
\[
\sum_{t=1}^{h-1}(m_t+1)=|S|+h-1
\]
length-two connectors, proving the length formula.
\end{proof}

\subsubsection{The direct terms}

\begin{lemma}\label{lem:conditional-direct}
For every fixed $\mathbf y\in\cO^{h-1}$,
\begin{equation}\label{eq:conditional-direct}
\sum_{j=0}^{h-1}\sum_{\mathbf x\in C_j(\mathbf y)}
|L_j(\mathbf y,\mathbf x)|
\le(h+1)\fall r h\,N.
\end{equation}
\end{lemma}

\begin{proof}
Partition the vertices $w$ in \eqref{eq:Lj-plus}--\eqref{eq:Lj-minus} according to $S=T(w)\in\mathscr S$.
Fix one such $S$ and put
\[
W=S\cup T_1\cup\cdots\cup T_{h-1},
\qquad m=|W\setminus\{p,q\}|.
\]
We count the possible extended kernel tuples for one vertex $w\in\cO_S$.

There are $h+1$ direct positions indexed by $\ell\in\{0,\ldots,h\}$.  For $0\le\ell\le h-1$, position $\ell$ corresponds to $L_\ell^+$.  Position $h$ corresponds to $L_{h-1}^-$.  A member at position $\ell<h$, with tuple $\mathbf x$ and additional kernel vertex $z$, determines the ordered internal kernel tuple
\[
(c_1,\ldots,c_h)=
(x_1,\ldots,x_\ell,z,x_{\ell+1},\ldots,x_{h-1}).
\]
For position $h$, put
\[
(c_1,\ldots,c_h)=(x_1,\ldots,x_{h-1},z).
\]
In every case the $c_i$ are distinct members of $W\setminus\{p,q\}$.  For a fixed position, the map to $\mathbf c$ is injective.

Set $c_0=p$ and $c_{h+1}=q$.  At position $\ell$, $c_\ell c_{\ell+1}$ is the kernel edge.  Every other consecutive pair is assigned one of the $h$ types
\[
T_1,\ldots,T_{h-1},S,
\]
and each type is used exactly once.  More explicitly, for $\ell<h$ the ordered type list is
\[
T_1,\ldots,T_\ell,S,T_{\ell+1},\ldots,T_{h-1}.
\]
types before the kernel edge are assigned to the preceding pair, and types after it are assigned to the following pair.  For position $h$, the list is $T_1,\ldots,T_{h-1},S$ and every type is assigned to the preceding pair.

Suppose such a direct path exists.  If $m\ge a$, choose $a-h$ vertices from
\[
W\setminus\bigl(\{p,q\}\cup\{c_1,\ldots,c_h\}\bigr).
\]
Assign each chosen vertex to one of the $h$ types that contains it and insert it into the corresponding connector.  The initial path has $h$ length-two connectors and one kernel edge.  Thus the new path has length $2h+1+2(a-h)=2a+1$.  It requests $a$ length-two connectors, so \cref{lem:connector-selection} makes their internal vertices distinct.  This contradicts the hypothesis.  Hence $m\le a-1=r+1$.

If $m\le r$, then each of the $h+1$ positions has at most $\fall m h\le\fall r h$ ordered kernel tuples, and hence their total number for the fixed type $S$ is at most
\begin{equation}\label{eq:direct-m-r}
(h+1)\fall r h.
\end{equation}

It remains to consider $m=r+1$.  We claim that a fixed ordered tuple
\[
\mathbf c=(c_1,\ldots,c_h)\in(W\setminus\{p,q\})_{\ne}^{h}
\]
cannot occur at two positions.  Suppose that positions $\ell<k$ both occur.  Keep the type assignment from position $\ell$ on all consecutive pairs except $c_\ell c_{\ell+1}$.  Since $k\ne\ell$, position $k$ assigns a type to that pair.  Thus every consecutive pair of
\[
p,c_1,\ldots,c_h,q
\]
is assigned one of the types $T_1,\ldots,T_{h-1},S$.  The assignment from position $\ell$ already uses every one of these types.  Insert every remaining vertex of $W\setminus\{p,q\}$ into a connector whose type contains it.  By \cref{lem:connector-selection}, the internal vertices can be chosen to be distinct.  The completed path has
\[
h+1+(m-h)=m+1=r+2=a
\]
length-two connectors and therefore length $2a$, a contradiction.  Hence each ordered kernel tuple occurs at most once, and their number is at most $\fall{r+1}h$.
Finally,
\[
\frac{\fall{r+1}h}{\fall r h}
=\frac{r+1}{r-h+1}
=\frac{a-1}{a-s}
\le h+1,
\]
where the last inequality is equivalent to $(s-1)(a-s-1)\ge0$.
Thus \eqref{eq:direct-m-r} also holds when $m=r+1$.

For the type $S$, every extended kernel tuple has at most $|\cO_S|$ choices for $w$.  Summing over $S\in\mathscr S$ and using
\[
\sum_S|\cO_S|\le N
\]
proves \eqref{eq:conditional-direct}.
\end{proof}

Recall that $b=|B|$, where $B=(T_1\cup\cdots\cup T_{h-1})\setminus\{p,q\}$.  Thus $b$ is the number of kernel vertices available in the types used by $\mathbf y$.  We split the estimate into the cases $b\le r-1$, $b=r$, and $b\ge r+1$.

\subsubsection{The case \texorpdfstring{$b\le r-1$}{b <= r - 1}}

Fix $j$.  Let $L_j^{\rm in}$ be the set of members of $L_j$ whose additional kernel vertex $z$ belongs to $B$, and put $L_j^{\rm out}=L_j\setminus L_j^{\rm in}$.

\begin{lemma}\label{lem:noncritical-local}
If $b\le r$, then, with $\lambda_j=1$ for $j<h-1$ and $\lambda_{h-1}=2$,
\begin{align}
&\sum_{\mathbf x\in C_j(\mathbf y)}
\bigl(|L_j(\mathbf y,\mathbf x)|+|E_j(\mathbf y,\mathbf x)|\bigr)\notag\\
&\qquad\le
\left\{
\lambda_j\fall b h+(2(r-b)+1)\fall b{h-1}
\right\}N+O_{a,h,\eps}(1). \label{eq:Phi-local}
\end{align}
\end{lemma}

\begin{proof}
Every $\mathbf x\in C_j(\mathbf y)$ uses $h-1$ distinct vertices of $B$.  Hence
\begin{align}
\sum_{\mathbf x\in C_j(\mathbf y)}|L_j^{\rm in}(\mathbf y,\mathbf x)|
&\le \lambda_j\bigl(b-(h-1)\bigr)\fall b{h-1}N\notag\\
&=\lambda_j\fall b h\,N. \label{eq:Lin-final}
\end{align}
Here $\lambda_j$ counts the possible positions of the kernel edge, $\fall b{h-1}$ counts the ordered internal kernel vertices already used, $b-(h-1)$ counts the remaining choice of $z$, and $N$ bounds the choice of the outside vertex.

Fix $\mathbf x\in C_j(\mathbf y)$ and form a rooted graph $R_{j,\mathbf y,\mathbf x}$ with roots $r_0=x_j$ and $r_1=x_{j+1}$.  For every member of $L_j^{\rm out}$, add the rooted path $r_0-z-w-r_1$ or $r_0-w-z-r_1$.  For every member of $E_j$, add $r_0-u-v-r_1$.  All internal vertices lie outside $B$, and every internal edge is the middle edge of one of these rooted paths of length three.

Put $L_0=2a-2b$.
If the rooted graph contained an $r_0$--$r_1$ path of length at least $L_0$, then \cref{lem:rooted-shortening} would give one of length $L_0$ or $L_0+1$.  Its length is bounded in terms of $a$.  By \cref{lem:connector-selection}, the length-two connectors in \cref{lem:connector-insertion} can be chosen outside this path.  Inserting all vertices of $B$ would then produce a $p$--$q$ path of length at least
\[
2b+L_0=2a,
\]
a contradiction.  Thus the rooted graph has no path of length at least $L_0$.  Applying \cref{lem:rooted-p3} with $D=L_0-1$ gives
\[
|L_j^{\rm out}(\mathbf y,\mathbf x)|+|E_j(\mathbf y,\mathbf x)|
\le(2a-2b-3)N+O_{a,h,\eps}(1).
\]
Since $2a-2b-3=2(r-b)+1$ and $|C_j(\mathbf y)|\le\fall b{h-1}$,
combining this estimate with \eqref{eq:Lin-final} proves \eqref{eq:Phi-local}.
\end{proof}

\begin{lemma}\label{lem:noncritical-arithmetic}
Let $h\ge3$, $\lambda\in\{1,2\}$, and $b\le r-1$.  Then
\begin{equation}\label{eq:Phi-arithmetic}
\lambda\fall b h+(2(r-b)+1)\fall b{h-1}
\le\lambda\fall r h.
\end{equation}
\end{lemma}

\begin{proof}
Put $t=r-b\ge1$.  The falling-factorial difference satisfies
\begin{align}
\fall r h-\fall b h
&=\sum_{u=b}^{r-1}\bigl(\fall{u+1}h-\fall u h\bigr)\notag\\
&=\sum_{u=b}^{r-1}h\fall u{h-1}
\ge ht\fall b{h-1}. \label{eq:fall-difference-final}
\end{align}
For $\lambda=1$,
\[
ht-(2t+1)=(h-2)t-1\ge0.
\]
For $\lambda=2$, $2ht\ge2t+1$.  Multiplying \eqref{eq:fall-difference-final} by $\lambda$ proves \eqref{eq:Phi-arithmetic}.
\end{proof}

Consequently, by \cref{lem:noncritical-arithmetic}, if $b\le r-1$, then
\begin{equation}\label{eq:noncritical-total}
\sum_{j=0}^{h-1}\sum_{\mathbf x\in C_j(\mathbf y)}
\bigl(|L_j|+|E_j|\bigr)
\le(h+1)\fall r h\,N+O_{a,h,\eps}(1).
\end{equation}

\subsubsection{The case \texorpdfstring{$b=r$}{b = r}}

Throughout this subsection, assume $b=r$.
For $1\le t\le h-1$, define the private part of $T_t$ by
\[
\operatorname{Priv}(T_t)=
\left(T_t\setminus\{p,q\}\right)
\setminus\bigcup_{\substack{1\le u\le h-1\\u\ne t}}T_u.
\]
These private parts are pairwise disjoint.

For an ordered injective tuple
\[
\mathbf x=(x_1,\ldots,x_{h-1})\in B_{\ne}^{h-1},
\]
put $M(\mathbf x)=B\setminus\{x_1,\ldots,x_{h-1}\}$.
By \eqref{eq:h-r-d} and the assumption $b=r$,
\begin{equation}\label{eq:omitted-size}
|M(\mathbf x)|=r-(h-1)=d.
\end{equation}
Let $\mathcal A(\mathbf x)=\{j\in\{0,\ldots,h-1\}:\mathbf x\in C_j(\mathbf y),\ E_j(\mathbf y,\mathbf x)\ne\varnothing\}$.

\begin{lemma}\label{lem:residual-one-position}
For every $j$ and every $\mathbf x\in C_j(\mathbf y)$,
\begin{equation}\label{eq:residual-per-position}
|E_j(\mathbf y,\mathbf x)|\le N+O_{a,h,\eps}(1).
\end{equation}
\end{lemma}

\begin{proof}
Use only the length-three rooted paths in the rooted graph from the proof of \cref{lem:noncritical-local}.  If it contained a root-to-root path of length at least $4$, \cref{lem:rooted-shortening} would give one of length $4$ or $5$.  Inserting all $b=r$ vertices of $B$ would then give a $p$--$q$ path of length at least $2r+4=2a$.
Thus no rooted path has length at least $4$.  Applying \cref{lem:rooted-p3} with $D=3$ proves \eqref{eq:residual-per-position}.
\end{proof}

\begin{lemma}\label{lem:private-obstruction}
Let $j<k$ and suppose $j,k\in\mathcal A(\mathbf x)$.  Then
\begin{equation}\label{eq:private-obstruction}
M(\mathbf x)\cap\operatorname{Priv}(T_k)\ne\varnothing.
\end{equation}
\end{lemma}

\begin{proof}
Suppose instead that $M(\mathbf x)\cap\operatorname{Priv}(T_k)=\varnothing$.  Use the $h-1$ length-two connectors determined by position $j$, together with one length-three connector from $x_j$ to $x_{j+1}$.  Since $k>j$, the $T_k$-connector joins $x_k$ to $x_{k+1}$, where $x_h=q$.  Delete this connector and use a length-three connector from $x_k$ to $x_{k+1}$.

Every vertex $z\in M(\mathbf x)$ can be assigned to a remaining type.  Indeed, if $z\notin T_k$, choose any $T_t$ containing $z$.  If $z\in T_k$, then $z\notin\operatorname{Priv}(T_k)$, so $z\in T_t$ for some $t\ne k$.  Insert all vertices of $M(\mathbf x)$ into their assigned connectors.

The two nonempty sets $E_j$ and $E_k$ come from retained length-three connector families.  By \cref{lem:connector-selection}, the two length-three connectors and all length-two connectors can be chosen with distinct internal vertices.  The resulting simple $p$--$q$ path contains $h-2+d$ length-two connectors and two length-three connectors.  Its length is
\[
2(h-2+d)+3+3=2(h+d+1)=2a,
\]
a contradiction.  Hence \eqref{eq:private-obstruction} holds.
\end{proof}

\begin{lemma}\label{lem:active-position-count}
For every $\mathbf x\in B_{\ne}^{h-1}$,
\begin{equation}\label{eq:active-count}
|\mathcal A(\mathbf x)|\le\min\{h,d+1\}.
\end{equation}
\end{lemma}

\begin{proof}
If $\mathcal A(\mathbf x)=\varnothing$, there is nothing to prove.  Otherwise let $j_0=\min\mathcal A(\mathbf x)$.  For every $k\in\mathcal A(\mathbf x)\setminus\{j_0\}$, \cref{lem:private-obstruction} gives $M(\mathbf x)\cap\operatorname{Priv}(T_k)\ne\varnothing$.  The private parts are pairwise disjoint, so different values of $k$ require different vertices of $M(\mathbf x)$.  By \eqref{eq:omitted-size},
\[
|\mathcal A(\mathbf x)|-1\le d.
\]
There are only $h$ positions, proving \eqref{eq:active-count}.
\end{proof}

Put
\[
\mu_{a,s}=\min\{s-1,a-s+1\}=\min\{h,d+1\}.
\]
By \cref{lem:residual-one-position,lem:active-position-count},
\begin{align}
\sum_{j=0}^{h-1}\sum_{\mathbf x\in C_j(\mathbf y)}
|E_j(\mathbf y,\mathbf x)|
&\le \mu_{a,s}|B_{\ne}^{h-1}|N+O_{a,h,\eps}(1)\notag\\
&=\mu_{a,s}\fall r{h-1}N+O_{a,h,\eps}(1). \label{eq:critical-residual-total}
\end{align}

If $b\ge r+1$, then every set $E_j(\mathbf y,\mathbf x)$ is empty.  Indeed, \cref{lem:connector-insertion} together with one length-three connector would give a path of length
$2r+5>2a$: use the $h-1$ prescribed length-two connectors and choose only enough vertices of $B\setminus\{x_1,\ldots,x_{h-1}\}$ to obtain $r+1$ length-two connectors.  Together with the length-three connector, this requests $r+2=a$ connectors, so \cref{lem:connector-selection} applies.

\subsubsection{The fixed-endpoint estimate}

\begin{theorem}\label{thm:mixed-master}
Fix integers $h=s-1\ge3$ and $a\ge h+2$.  For every $\eps>0$, there are constants $N_0=N_0(a,h,\eps)$ and $C=C(a,h,\eps)$ such that the following holds.  Let $N\ge N_0$, let $J$ be an $N$-vertex graph with $\circum(J)\le2a+1$, and suppose that $J$ has no $p$--$q$ path of length at least $2a$.  Apply \cref{lem:kernel-normal-form} with $M=20(a+h+1)$, and use the resulting sets $K$, $\mathscr P$, $\mathscr S$, and the families $\mathsf R(c,c')$ in the canonical encoding above.  For every fixed $\mathbf y\in\cO^{h-1}$,
\begin{align}
&\sum_{j=0}^{h-1}\sum_{\mathbf x\in C_j(\mathbf y)}
\bigl(|L_j(\mathbf y,\mathbf x)|+|E_j(\mathbf y,\mathbf x)|\bigr)\notag\\
&\qquad\le
\left[(h+1)\fall r h+\mu_{a,s}\fall r{h-1}\right]N+O_{a,h,\eps}(1). \label{eq:mixed-local-final}
\end{align}
Consequently,
\begin{equation}\label{eq:mixed-global-final}
P_{2h+1}(J;p,q)
\le
\left[(h+1)\fall r h+\mu_{a,s}\fall r{h-1}+\eps\right]N^h
+CN^{h-1}.
\end{equation}
\end{theorem}

\begin{proof}
If $b\le r-1$, use \eqref{eq:noncritical-total}.  If $b=r$, combine the direct estimate \eqref{eq:conditional-direct} with \eqref{eq:critical-residual-total}.  If $b\ge r+1$, all sets $E_j$ are empty, and \cref{lem:conditional-direct} gives the explicit bound
\[
\sum_{j=0}^{h-1}\sum_{\mathbf x\in C_j(\mathbf y)}
\bigl(|L_j(\mathbf y,\mathbf x)|+|E_j(\mathbf y,\mathbf x)|\bigr)
\le(h+1)\fall r hN.
\]
These three bounds prove \eqref{eq:mixed-local-final}.  There are at most $N^{h-1}$ ordered tuples $\mathbf y$.  Thus summing \eqref{eq:mixed-local-final} contributes the coefficient in \eqref{eq:mixed-global-final} times $N^h$, while the $O_{a,h,\eps}(1)$ error for each tuple contributes $O_{a,h,\eps}(N^{h-1})$.  Adding the $\eps N^h+CN^{h-1}$ error in \eqref{eq:canonical-majorization} gives \eqref{eq:mixed-global-final}.
\end{proof}

\begin{lemma}\label{lem:strict-slack}
For every $a>s\ge4$,
\begin{equation}\label{eq:strict-slack}
\mu_{a,s}\fall r{h-1}
<2(h+1)\fall{r-1}{h-1}.
\end{equation}
\end{lemma}

\begin{proof}
By
\[
\fall r{h-1}=r\fall{r-1}{h-2},
\qquad
\fall{r-1}{h-1}=d\fall{r-1}{h-2},
\]
it is enough to prove
\begin{equation}\label{eq:slack-reduced}
\mu_{a,s}r<2(h+1)d.
\end{equation}
If $d\le h-1$, then $\mu_{a,s}=d+1$ and
\begin{align*}
2(h+1)d-(d+1)r
&=2(h+1)d-(d+1)(h+d-1)\\
&=(d-1)(h-d+1)+2>0.
\end{align*}
If $d\ge h-1$, then $\mu_{a,s}=h$ and
\begin{align*}
2(h+1)d-hr
&=2(h+1)d-h(h+d-1)\\
&=d(h+2)-h(h-1)\\
&\ge2(h-1)>0.
\end{align*}
Thus \eqref{eq:slack-reduced}, and hence \eqref{eq:strict-slack}, holds.
\end{proof}

\begin{corollary}\label{cor:odd-fixed-slack}
Fix $a>s\ge4$, and put $h=s-1$ and $r=a-2$.  For every $\eps>0$, there are constants $N_0$ and $C$ such that the following holds for every $N\ge N_0$.  If $J$ is an $N$-vertex graph with $\circum(J)\le2a+1$ and no $p$--$q$ path of length at least $2a$, then
\[
P_{2h+1}(J;p,q)
\le\left[(h+1)\fall r h+2(h+1)\fall{r-1}{h-1}+\eps\right]N^h
+CN^{h-1}.
\]
\end{corollary}

\begin{proof}
Combine \cref{thm:mixed-master} with \cref{lem:strict-slack}.
\end{proof}

\subsection{Odd cycles with an even circumference bound}\label{sec:even-fixed}

The even case has no critical length-three term, so the fixed-endpoint estimate is shorter.

\begin{theorem}\label{thm:even-fixed-endpoint}
Let $h\ge2$ and $a\ge h+2$.  For every $\eps>0$ there is a constant $C=C(a,h,\eps)$ such that the following holds for all sufficiently large $N$.  Let $J$ be an $N$-vertex graph with $\circum(J)\le2a$, and let $p\ne q$ be vertices such that $J$ has no $p$--$q$ path of length at least $2a-1$.  Then
\[
P_{2h+1}(J;p,q)
\le\bigl((h+1)\fall{a-2}{h}+\eps\bigr)N^h+CN^{h-1}.
\]
\end{theorem}

\begin{proof}
Apply \cref{lem:kernel-normal-form,lem:canonical-reduction} with $\kappa=0$ and $M=20(a+h+1)$.  Fix $\mathbf y\in\cO^{h-1}$, and define $T_t$, $B$, $b$, and $\lambda_j$ as in the preceding section.  Put $r=a-2$.

Fix $j$ and split $L_j=L_j^{\rm in}\sqcup L_j^{\rm out}$ according as the additional kernel vertex belongs to $B$.  The same direct count as in \cref{lem:noncritical-local} gives
\[
\sum_{\mathbf x\in C_j(\mathbf y)}|L_j^{\rm in}(\mathbf y,\mathbf x)|
\le\lambda_j\fall b hN.
\]
This term is zero when $b\ge r+1$.  Indeed, choose only enough vertices of $B$ to obtain $r$ length-two connectors.  Together with the remaining length-two connector and the kernel edge, they give a path of length $2r+3=2a-1$.  At most $r+1=a-1$ connectors are requested.

Assume first that $b\le r-1$.  For fixed $\mathbf x$, put all paths represented by $L_j^{\rm out}$ and $E_j$ in the rooted graph used in \cref{lem:noncritical-local}.  If this graph contained a root-to-root path of length at least $L_0=2a-2b-1$,
then \cref{lem:rooted-shortening} and \cref{lem:connector-insertion,lem:connector-selection} would give a $p$--$q$ path of length at least $2b+L_0=2a-1$.  Hence no such rooted path exists.  Applying \cref{lem:rooted-p3} with $D=L_0-1$ gives
\[
|L_j^{\rm out}(\mathbf y,\mathbf x)|+|E_j(\mathbf y,\mathbf x)|
\le2(r-b)N+O_{a,h,\eps}(1).
\]
When $b\ge r$, the same left side is zero: choose enough vertices of $B$ to obtain $r$ length-two connectors and add the given length-three connector.  The resulting path has length $2r+3=2a-1$ and requests $r+1=a-1$ connectors.

It follows that, for $b\le r$,
\[
\sum_{\mathbf x\in C_j(\mathbf y)}
\bigl(|L_j|+|E_j|\bigr)
\le\left[\lambda_j\fall b h+2(r-b)\fall b{h-1}\right]N+O_{a,h,\eps}(1).
\]
Moreover,
\[
\fall r h-\fall b h
=\sum_{u=b}^{r-1}h\fall u{h-1}
\ge h(r-b)\fall b{h-1}.
\]
Since $h\ge2$ and $\lambda_j\in\{1,2\}$, the last two displays imply
\[
\lambda_j\fall b h+2(r-b)\fall b{h-1}
\le\lambda_j\fall r h.
\]
The same inequality is trivial when $b>r$, because all terms vanish.  Summing over $j$ and using $\sum_j\lambda_j=h+1$ yields
\[
\sum_{j=0}^{h-1}\sum_{\mathbf x\in C_j(\mathbf y)}
\bigl(|L_j|+|E_j|\bigr)
\le(h+1)\fall r hN+O_{a,h,\eps}(1).
\]
Finally sum over $\mathbf y$ and use the canonical encoding.  This proves the theorem.
\end{proof}

\subsection{Odd cycles when \texorpdfstring{$s=3$}{s = 3}}\label{sec:s3}

We first state the path estimate needed in the deletion argument.

\begin{lemma}\label{lem:s3-rooted-paths}
Fix $a\ge4$ and $\eps>0$, and put $r=a-2$.  There are constants $N_0=N_0(a,\eps)$ and $C=C(a,\eps)$ such that the following holds.  Let $J$ be an $N$-vertex graph with $N\ge N_0$ and $\circum(J)\le2a+1$, and let $U\subseteq V(J)$ satisfy $|U|\le a-1$.  If $J$ has no path of length at least $2a$ with both endpoints in $U$, then
\[
P_5(J;U)
\le \binom{|U|}{2}(3r^2-r+\eps)N^2+CN.
\]
\end{lemma}

\begin{proof}
If $|U|\le1$, then $P_5(J;U)=0$.  Assume $|U|\ge2$.  Put $\kappa=1$ and $h=2$, and fix distinct $p,q\in U$.  Then $J$ has no $p$--$q$ path of length at least $2a$.  Apply \cref{lem:kernel-normal-form,lem:canonical-reduction} with the given value of $\eps$ and with $M=20(a+3)$.  Define
\[
\xi_A=\frac{|\cO_A|}{N}\quad(A\in\mathscr S),
\qquad \theta=\sum_{A\in\mathscr S}\xi_A\le1.
\]
For $A,B\in\mathscr S$, put $X=A\setminus\{p,q\}$, $Y=B\setminus\{p,q\}$,
\[
x=|X|,\quad y=|Y|,\quad z=|X\cap Y|,
\quad u=|X\cup Y|,
\quad \alpha=\ind_{\{p\in A\}},\quad \beta=\ind_{\{q\in B\}}.
\]
Define the three direct terms by
\begin{align*}
D_0(A,B)&=\beta\bigl|\{(z_0,c)\in X\times(X\cap Y):z_0\ne c,\ pz_0\in E(J[K])\}\bigr|,\\
D_1(A,B)&=\alpha\beta\bigl|\{(c,z_0)\in X\times Y:c\ne z_0,\ cz_0\in E(J[K])\}\bigr|,\\
D_2(A,B)&=\alpha\bigl|\{(c,z_0)\in(X\cap Y)\times Y:c\ne z_0,\ z_0q\in E(J[K])\}\bigr|.
\end{align*}
The two normalized length-three terms are
\begin{align*}
G^L(B)&=\frac{\beta}{N}\sum_{c\in Y}
\bigl|\{(u_0,v_0)\in\cO^2:(p,u_0,v_0,c)\in\mathsf R(p,c)\}\bigr|,\\
G^R(A)&=\frac{\alpha}{N}\sum_{c\in X}
\bigl|\{(u_0,v_0)\in\cO^2:(c,u_0,v_0,q)\in\mathsf R(c,q)\}\bigr|.
\end{align*}

We now explain the five terms.  Let $w_A\in\cO_A$ and $w_B\in\cO_B$.  Up to the error in \cref{lem:kernel-normal-form}, every $p$--$q$ path of length five has exactly one of the following forms:
\begin{enumerate}[label=(\roman*),leftmargin=2.4em]
\item $p z_0 w_A c w_B q$, where $pz_0\in E(J[K])$, $z_0,c\in X$, $c\in Y$, and $q\in B$.  Thus $z_0\ne c$, and the number of possible ordered pairs $(z_0,c)$ is $D_0(A,B)$.
\item $p w_A c z_0 w_B q$, where $p\in A$, $q\in B$, $c\in X$, $z_0\in Y$, and $cz_0\in E(J[K])$.  The number of possible ordered pairs $(c,z_0)$ is $D_1(A,B)$.
\item $p w_A c w_B z_0 q$, where $p\in A$, $c\in X\cap Y$, $z_0\in Y$, and $z_0q\in E(J[K])$.  Thus $c\ne z_0$, and the number of possible ordered pairs $(c,z_0)$ is $D_2(A,B)$.
\item $p u_0v_0c w_Bq$, where $(p,u_0,v_0,c)\in\mathsf R(p,c)$, $c\in Y$, and $q\in B$.  These paths are counted by $G^L(B)$ after normalization.
\item $p w_Ac u_0v_0q$, where $p\in A$, $c\in X$, and $(c,u_0,v_0,q)\in\mathsf R(c,q)$.  These paths are counted by $G^R(A)$ after normalization.
\end{enumerate}
For fixed $A,B$, there are at most $|\cO_A||\cO_B|=\xi_A\xi_BN^2$ ordered choices of $(w_A,w_B)$ in cases~(i)--(iii).  Possible coincidences only cause overcounting and are harmless for an upper bound.  In case~(iv), the factor $|\cO_B|=\xi_BN$ counts $w_B$, while the definition of $G^L(B)$ divides the number of length-three connectors by $N$.  After division by $N^2$, the contribution is $\xi_BG^L(B)$.  The same argument gives $\xi_AG^R(A)$ in case~(v).  Therefore the canonical encoding gives
\begin{align}\label{eq:s3-average-interface}
\frac{P_5(J;p,q)}{N^2}
\le{}&
\sum_{A,B\in\mathscr S}\xi_A\xi_B
\bigl(D_0(A,B)+D_1(A,B)+D_2(A,B)\bigr)\nonumber\\
&+\sum_{B\in\mathscr S}\xi_B G^L(B)
+\sum_{A\in\mathscr S}\xi_A G^R(A)
+\eps+O_{a,\eps}(N^{-1}).
\end{align}
The term $\eps$ accounts for the paths discarded in \cref{lem:kernel-normal-form}, and the term $O_{a,\eps}(N^{-1})$ comes from the $O_{a,\eps}(N)$ remaining error after division by $N^2$.

The definitions give
\begin{equation}\label{eq:direct-s3-bounds}
\begin{aligned}
D_0(A,B)&\le\beta z(x-1),&
D_1(A,B)&\le\alpha\beta(xy-z),&
D_2(A,B)&\le\alpha z(y-1).
\end{aligned}
\end{equation}
For $D_0$, the vertex $c$ lies in $X\cap Y$ and $z_0$ is a different member of $X$.  The estimate for $D_2$ is symmetric.  For $D_1$, the number of ordered pairs of distinct vertices in $X\times Y$ is $xy-z$.

Define
\begin{equation}\label{eq:g-def}
g(t)=
\begin{cases}
(2(r-t)+1)t,&0\le t\le r,\\
0,&t\ge r+1.
\end{cases}
\end{equation}
Then
\begin{equation}\label{eq:residual-s3}
G^L(B)\le\beta g(y),
\qquad
G^R(A)\le\alpha g(x).
\end{equation}
We prove the first inequality.  The second is symmetric.  Fix $c\in Y$.  If $y\ge r+1$ and a length-three connector from $p$ to $c$ exists, choose $r$ vertices of $Y\setminus\{c\}$.  They, together with $c$ and $q$, give $r+1$ length-two connectors.  Adding the length-three connector gives a path of length $2(r+1)+3>2a$.  This construction requests $r+2=a$ connectors, so \cref{lem:connector-selection} applies.  Hence this contribution is zero when $y\ge r+1$.

We may now assume $y\le r$.  Since every vertex of $Y\cup\{q\}$ belongs to $B$, \cref{lem:connector-selection} gives a simple path from $c$ through all vertices of $Y\setminus\{c\}$ to $q$ made of length-two connectors.  Its length is $2y$.  Consider the union of the length-three connectors from $p$ to $c$ counted in the definition of $G^L(B)$.  Every internal edge of this rooted graph is the middle edge of one of its defining rooted paths of length three.  If it contained a $p$--$c$ path of length at least
\[
L_0=2a-2y,
\]
then \cref{lem:rooted-shortening} would give one of length $L_0$ or $L_0+1$.  By \cref{lem:connector-selection}, the length-two connectors can be chosen outside that bounded path, producing a forbidden $p$--$q$ path of length at least $2a$.  Hence \cref{lem:rooted-p3}, with $D=L_0-1$, gives at most
\[
(2a-2y-3)N=(2(r-y)+1)N
\]
length-three connectors for this $c$.  Summing over the $y$ choices for $c$ proves \eqref{eq:residual-s3}.

We prove that the following two constraints hold:
\begin{align}
\alpha\beta z>0&\quad\Longrightarrow\quad u\le r,\label{eq:support-A}\\
D_0+D_1+D_2>0&\quad\Longrightarrow\quad u\le r+1.\label{eq:support-B}
\end{align}
For \eqref{eq:support-A}, choose $c\in X\cap Y$.  If $u\ge r+1$, choose a set $W\subseteq X\cup Y$ of order $r+1$ that contains $c$.  Order the vertices of $W\cap X$ from $p$ to $c$, and then order the remaining vertices of $W$ from $c$ to $q$.  A vertex of $X\cap Y$ may be put on either side.  The connector-selection lemma gives a path with $r+2=a$ length-two connectors, and hence with length $2a$, a contradiction.

For \eqref{eq:support-B}, use the kernel edge in the nonzero direct term.  If $u\ge r+2$, select $r+2=a$ vertices of $X\cup Y$, including the kernel vertices required by that direct term, and put the others into the corresponding $A$- or $B$-connector.  This gives a path of length $2(r+2)+1>2a$, a contradiction.  Only $a$ length-two connectors are requested.

\begin{claim}\label{clm:s3-pointwise}
For every $A,B\in\mathscr S$,
\[
D_0+D_1+D_2+G^L+G^R\le3r^2-r.
\]
\end{claim}

\begin{proof}
Fix $A,B\in\mathscr S$ and omit them from the notation.

\smallskip
\noindent\emph{Case 1: $\alpha=\beta=1$ and $z>0$.}
By \eqref{eq:support-A}, $u\le r$.  Interchange $A,B$ if necessary and assume $x\le y$.  If $x+y=2$, then $x=y=z=1$, the direct contribution is zero, and
\[
g(x)+g(y)=4r-2\le3r^2-r.
\]
Assume $x+y\ge3$.  From \eqref{eq:direct-s3-bounds},
\[
D_0+D_1+D_2
\le xy+z(x+y-3)
\le x^2+2xy-3x.
\]
Using \eqref{eq:g-def}, the total is at most
\begin{equation}\label{eq:case1quad}
-x^2+2xy-2y^2+(2r-2)x+(2r+1)y.
\end{equation}
For fixed $y$, the derivative of \eqref{eq:case1quad} with respect to $x$ is
\[
-2x+2y+2r-2>0
\]
on $0\le x\le y\le r$.  Its maximum therefore occurs at $x=y$, where it is
\[
-y^2+(4r-1)y.
\]
This expression is increasing on $0\le y\le r$, and its value at $y=r$ is $3r^2-r$.

\smallskip
\noindent\emph{Case 2: $\alpha=\beta=1$ and $z=0$.}
Only the middle direct term can be nonzero, so $D_0+D_1+D_2\le xy$.  Suppose first that it is nonzero.  By \eqref{eq:support-B},
\[
x+y\le r+1.
\]
If $x=r+1$ or $y=r+1$, the other variable is zero, so the whole contribution is zero.  We may therefore assume $x,y\le r$.  Put $v=x+y$.  Then
\begin{align*}
xy+g(x)+g(y)
&=(2r+1)v+xy-2(x^2+y^2)\\
&=(2r+1)v+5xy-2v^2\\
&\le(2r+1)v-\frac34v^2.
\end{align*}
The last expression is increasing for $0\le v\le r+1$, because its derivative is at least $(r-1)/2$.  Hence
\[
xy+g(x)+g(y)
\le\frac{(r+1)(5r+1)}4
\le3r^2-r,
\]
where the last inequality is equivalent to $7r^2-10r-1\ge0$.

If the middle direct term is zero, each nonzero value of $g$ lies on $[0,r]$, and
\[
\max_{0\le t\le r}g(t)\le\frac{(2r+1)^2}{8}.
\]
Thus
\[
g(x)+g(y)\le\frac{(2r+1)^2}{4}\le3r^2-r.
\]

\smallskip
\noindent\emph{Case 3: exactly one of $\alpha,\beta$ equals $1$.}
Assume $\alpha=1$ and $\beta=0$.  Then
\[
D_0+D_1+D_2\le z(y-1),
\qquad
G^L+G^R\le g(x).
\]
If a direct term is nonzero, \eqref{eq:support-B} gives $u\le r+1$.  Since $y\le u$, we have $y-1\le r$, and hence
\[
z(y-1)\le rx.
\]
For $x\le r$,
\[
D_0+D_1+D_2+G^L+G^R
\le(3r+1)x-2x^2
\le\frac{(3r+1)^2}{8}
\le3r^2-r.
\]
If $x=r+1$, then $g(x)=0$ and the direct part is at most $r(r+1)\le3r^2-r$.  If $x>r+1$, \eqref{eq:support-B} forces the direct part to be zero and \eqref{eq:g-def} gives $g(x)=0$.  When all direct terms are zero, the same conclusion follows from $g(x)\le(2r+1)^2/8$.  The case $\alpha=0$, $\beta=1$ is symmetric.

\smallskip
\noindent\emph{Case 4: $\alpha=\beta=0$.}
All five terms vanish.
\end{proof}

If $\theta=0$, \eqref{eq:s3-average-interface} immediately gives the next estimate.  Suppose $\theta>0$ and put $p_A=\xi_A/\theta$ for $A\in\mathscr S$.  The three direct terms in \eqref{eq:s3-average-interface} have total weight $\theta^2$, while the two length-three terms have total weight $\theta$.  Since all five terms are nonnegative and $\theta\le1$, \cref{clm:s3-pointwise} gives
\[
\theta^2\sum_{A,B}p_Ap_B(D_0+D_1+D_2)
+\theta\sum_{A,B}p_Ap_B(G^L+G^R)
\le \theta\sum_{A,B}p_Ap_B(D_0+D_1+D_2+G^L+G^R)
\le3r^2-r.
\]
Here all sums are over $A,B\in\mathscr S$.
Together with \eqref{eq:s3-average-interface}, this gives
\begin{equation}\label{eq:s3-master}
P_5(J;p,q)\le(3r^2-r+\eps)N^2+O_{a,\eps}(N).
\end{equation}
Since $p,q$ were arbitrary distinct vertices of $U$, summing over the $\binom{|U|}{2}$ unordered endpoint pairs proves the lemma after increasing $C$.
\end{proof}

For later use, the relaxed per-pair target is
\[
\frac{r+2}{r}\,3\fall r2=3(r+2)(r-1),
\]
and
\begin{equation}\label{eq:s3-gap}
3(r+2)(r-1)-(3r^2-r)=4r-6>0.
\end{equation}
This is the positive gap needed in the low-degree deletion estimate when $s=3$.

\subsection{The low-degree deletion inequality}\label{sec:deletion}

We now convert the fixed-endpoint estimates into deletion inequalities.  We distinguish the parity of the counted cycle and treat even cycles first.

\begin{lemma}\label{lem:even-low-degree-deletion}
Fix $a\ge s\ge3$ and $\kappa\in\{0,1\}$.  There are constants $n_1$ and $\gamma>0$ such that the following holds.  Let $G$ be an $n$-vertex graph with $n\ge n_1$ and $\circum(G)\le2a+\kappa$, and let $v\in V(G)$ have degree at most $a-1$.  Then
\[
N(C_{2s},G)-N(C_{2s},G-v)
\le E_{a,s}^{(\kappa)}(n)-E_{a,s}^{(\kappa)}(n-1)-\gamma n^{s-1}.
\]
\end{lemma}

\begin{proof}
Put $J=G-v$, $U=N_G(v)$, and $\rho=|U|\le a-1$.  The Erd\H{o}s--Gallai cycle theorem~\cite{ErdosGallai1959} gives $e(J)=O_a(n)$.  Moreover, $J$ has no path of length $2a$ with both endpoints in $U$, since this path together with the two incident edges at $v$ would give a cycle of length $2a+2>2a+\kappa$.

Every copy of $C_{2s}$ containing $v$ is obtained uniquely from an unoriented path of length $2s-2$ in $J$ whose two endpoints lie in $U$.  Apply \cref{lem:even-rooted-paths} with $r=s-1$ to obtain
\[
N(C_{2s},G)-N(C_{2s},G-v)
\le\binom\rho2\fall{a-2}{s-2}n^{s-1}+o_{a,s}(n^{s-1}).
\]
By \eqref{eq:E-diff}, the leading coefficient of the candidate increment is $\binom a2\fall{a-2}{s-2}$.  Since $\rho\le a-1$, the difference between these two coefficients is at least $(a-1)\fall{a-2}{s-2}>0$.  Taking $\gamma$ smaller than half this number and then increasing $n_1$ proves the lemma.
\end{proof}

For $a>s\ge3$ and $\kappa\in\{0,1\}$, put $\mu_{a,s}=\min\{s-1,a-s+1\}$ and
\[
C_{a,s}^{(\kappa)}
=s\fall{a-2}{s-1}
+\kappa\mu_{a,s}\fall{a-2}{s-2}.
\]
The fixed-endpoint estimates give $C_{a,s}^{(\kappa)}$ as the coefficient of $N^{s-1}$: for $\kappa=0$ this is \cref{thm:even-fixed-endpoint}, for $\kappa=1$ and $s\ge4$ it is \cref{thm:mixed-master}, and for $\kappa=1$ and $s=3$ it follows from \cref{lem:s3-rooted-paths} because
\[
3r^2-r=3\fall r2+2\fall r1,
\qquad \mu_{a,3}=2.
\]

\begin{lemma}\label{lem:odd-low-degree-deletion}
Fix $a>s\ge3$ and $\kappa\in\{0,1\}$.  There are constants $n_1=n_1(a,s,\kappa)$ and $\gamma_{a,s}^{(\kappa)}>0$ such that the following holds.  Let $G$ be an $n$-vertex graph with $n\ge n_1$ and $\circum(G)\le2a+\kappa$, and let $v\in V(G)$ satisfy $d_G(v)\le a-1$.  Then
\begin{equation}\label{eq:deletion-increment}
N(C_{2s+1},G)-N(C_{2s+1},G-v)
\le F_{a,s}^{(\kappa)}(n)-F_{a,s}^{(\kappa)}(n-1)-\gamma_{a,s}^{(\kappa)}n^{s-1}.
\end{equation}
\end{lemma}

\begin{proof}
Put $h=s-1$ and $r=a-2$.
The coefficient available after summing over endpoint pairs is
\[
\frac a{a-2}s\fall{a-2}{s-1}
=\frac{r+2}{r}(h+1)\fall r h
=(h+1)\fall r h+2(h+1)\fall{r-1}{h-1}.
\]
Moreover, $C_{a,s}^{(\kappa)}$ is strictly smaller than this number.  For $\kappa=0$, the difference is $2s\fall{a-3}{s-2}>0$.  For $\kappa=1$ and $s\ge4$, this follows from \cref{lem:strict-slack}.  For $\kappa=1$ and $s=3$, it follows from \eqref{eq:s3-gap}.  Thus
\[
\eta_{a,s}^{(\kappa)}
\defeq\frac a{a-2}s\fall{a-2}{s-1}-C_{a,s}^{(\kappa)}>0.
\]

Put
\[
J=G-v,\qquad U=N_G(v),\qquad d_v=|U|\le a-1.
\]
Every copy of $C_{2s+1}$ containing $v$ is obtained uniquely by adjoining the two edges from $v$ to the endpoints of a path of length $2s-1$ in $J$ with distinct endpoints in $U$.  The graph $J$ contains no path of length at least $2a+\kappa-1$ with both endpoints in $U$, since such a path together with the two edges through $v$ would give a cycle of length at least $2a+\kappa+1$.

Choose the error parameter in the applicable fixed-endpoint estimate to be at most $\eta_{a,s}^{(\kappa)}/2$.  For $\kappa=0$, apply \cref{thm:even-fixed-endpoint} to every pair of distinct vertices in $U$.  For $\kappa=1$ and $s\ge4$, use \cref{thm:mixed-master} in the same way.  For $\kappa=1$ and $s=3$, apply \cref{lem:s3-rooted-paths} directly to $U$.  In the last case, the coefficient in that lemma is at most $C_{a,3}^{(1)}+\eta_{a,3}^{(1)}/2$, which equals the displayed coefficient below.  Thus, after summing over the unordered endpoint pairs when needed and using $\binom{a-1}{2}\frac a{a-2}=\binom a2$,
\begin{align}
&N(C_{2s+1},G)-N(C_{2s+1},G-v)\notag\\
&\quad\le
\binom{d_v}{2}
\left[
\frac a{a-2}s\fall{a-2}{s-1}-\frac{\eta_{a,s}^{(\kappa)}}2
\right]n^{s-1}+O_{a,s}(n^{s-2})\notag\\
&\quad\le
\left[
s\binom a2\fall{a-2}{s-1}
-\binom{a-1}{2}\frac{\eta_{a,s}^{(\kappa)}}2
\right]n^{s-1}+O_{a,s}(n^{s-2}). \label{eq:deletion-leading-reserve}
\end{align}
Put
\[
\rho_{a,s}^{(\kappa)}=\binom{a-1}{2}\frac{\eta_{a,s}^{(\kappa)}}2>0.
\]
The coefficient in \eqref{eq:deletion-leading-reserve} is smaller than the leading coefficient in \eqref{eq:F-diff} by $\rho_{a,s}^{(\kappa)}$.  The remaining error is $O_{a,s}(n^{s-2})$.  Hence, after increasing $n_1$, \eqref{eq:deletion-increment} holds with, for example,
\[
\gamma_{a,s}^{(\kappa)}=\frac{\rho_{a,s}^{(\kappa)}}2.
\]
\end{proof}
\section{High minimum degree and peeling}\label{sec:terminal}

The deletion inequalities in the preceding section apply while a vertex has degree at most $a-1$.  We now treat the remaining case of minimum degree at least $a$.  For a large two-connected graph, the stability theorem below reduces the problem to the two model graphs and four exceptional families.  We bound the exceptional families first and then obtain the required high-minimum-degree estimate.

We use the following convenient consequence of Theorem~4 of Zhu, Gy\H{o}ri, He, Lv, Salia and Xiao~\cite{ZhuEtAl2023}.  The formulation below groups the exceptional graphs in their theorem into the broader families \textup{(E1)}--\textup{(E4)}, which are sufficient for our counting arguments.

\begin{theorem}[Zhu, Gy\H{o}ri, He, Lv, Salia and Xiao~\cite{ZhuEtAl2023}]\label{thm:stability}
Let $G$ be a $2$-connected $n$-vertex graph with $n\ge\circum(G)+1$ and $\delta(G)=k\ge3$.  Then either $\circum(G)\ge2k+2$, or $G$ is a subgraph of $H(n,2k)$ or $H(n,2k+1)$, or one of the following holds:
\begin{enumerate}[label=(E\arabic*),leftmargin=2.7em]
\item There are distinct vertices $u,u_1,u_2$ such that every component $Q$ of $G-\{u,u_1,u_2\}$ has order at most $k-1$ and, for some $i\in\{1,2\}$, no vertex of $Q$ is adjacent to a vertex of $\{u,u_1,u_2\}\setminus\{u,u_i\}$.
\item There is a set $Z\subseteq V(G)$ with $|Z|=2$ such that every component of $G-Z$ has order at most $k$.
\item We have $k=4$ and $G\subseteq K_3\vee bK_2$ for some $b\ge0$.
\item We have $k=3$ and $G\subseteq K_2\vee(S_r\cup tK_2)$ for some $r,t\ge0$, where $S_r$ is a star with $r$ leaves and $r+2t=n-3$.
\end{enumerate}
\end{theorem}

\begin{lemma}\label{lem:exceptional-count}
Fix $a>s\ge3$ and $\kappa\in\{0,1\}$.  Let $E_n$ be an $n$-vertex graph in one of the exceptional families in \cref{thm:stability} with $k=a$.  If $n$ is sufficiently large, then
\[
N(C_{2s+1},E_n)\le F_{a,s}^{(\kappa)}(n).
\]
\end{lemma}

\begin{proof}
Let $Z$ be the distinguished set in the description of the relevant family.  For a cycle $C$, every component of $C-Z$ is a path contained in one component of $E_n-Z$.

For family~\textup{(E1)}, let $Z=\{u,u_1,u_2\}$.  Every component of $E_n-Z$ is adjacent either only to $\{u,u_1\}$ or only to $\{u,u_2\}$.  If a path component of $C-Z$ has both ends adjacent to the same vertex of $Z$, then it uses both cycle edges at that vertex and $C$ meets no other component of $E_n-Z$.  Otherwise every path of the first kind joins $u$ to $u_1$, and every path of the second kind joins $u$ to $u_2$.  Since $C$ uses only two edges at $u$, it meets at most two components of $E_n-Z$.  There are $O(n^2)$ ways to choose these components.  Their orders are bounded in terms of $a$, so after they are chosen there are only $O_{a,s}(1)$ possible vertex sets and cyclic orders.  Hence $N(C_{2s+1},E_n)=O_{a,s}(n^2)$.

For family~\textup{(E2)}, deletion of the two vertices in $Z$ leaves components of bounded order.  The graph $C-Z$ has at most two path components, and hence $C$ meets at most two components of $E_n-Z$.  The same count gives $O_{a,s}(n^2)$ copies.

The family $K_3\vee bK_2$ occurs only for $a=4$.  Let $Z$ be its $K_3$.  The graph $C-Z$ has at most three path components, so $C$ uses vertices from at most three of the $K_2$ components.  Consequently every fixed odd cycle has $O_s(n^3)$ copies.

If $s\ge4$, these estimates are $o(n^s)$, whereas \eqref{eq:F-asymp} is $\Theta_{a,s}(n^s)$.  It remains to treat $s=3$ in $K_3\vee bK_2$.  Here $n=2b+3$.  If a $C_7$ used at most two vertices of $Z$, then $C-Z$ would have at most two path components, each with at most two vertices, and the cycle would have order at most six.  Thus every $C_7$ uses all three vertices of $Z$ and four vertices outside $Z$.  The orders of the path components of $C-Z$ are therefore either $2,2$ or $2,1,1$.  In the first case, choose the unused pair of consecutive vertices of $Z$, the two ordered $K_2$ components, and the direction of each of the two edges.  This gives at most $12b(b-1)$ cycles.  In the second case, choose which of the three paths has order two, its $K_2$ component, the components containing the two single vertices, and the vertex or direction used in each component.  The two single vertices may come from the same $K_2$.  This gives at most $24b(b-1)^2$ cycles.  Thus
\[
N(C_7,K_3\vee bK_2)\le12b(b-1)+24b(b-1)^2=24b^3+O(b^2).
\]
On the other hand,
\[
F_{4,3}^{(\kappa)}(2b+3)=12(2b+3)^3+O(b^2)=96b^3+O(b^2)
\]
by \eqref{eq:F-asymp}.  The desired inequality follows for large $b$.
\end{proof}

\begin{lemma}\label{lem:even-exceptional-count}
Fix $a\ge s\ge3$ and $\kappa\in\{0,1\}$.  If an $n$-vertex graph $G$ belongs to an exceptional family in \cref{thm:stability} with $k=a$, then, for all sufficiently large $n$,
\[
N(C_{2s},G)\le E_{a,s}^{(\kappa)}(n).
\]
\end{lemma}

\begin{proof}
In the first two exceptional families, a cycle meets at most two components outside the fixed set, so they contain $O_{a,s}(n^2)$ copies of $C_{2s}$.  In family~\textup{(E3)}, a cycle meets at most three of the $K_2$ components, giving $O_s(n^3)$ copies.  Family~\textup{(E4)} occurs only when $a=3$.  Thus \eqref{eq:E-asymp} proves the result when $s\ge4$.

Let $s=3$.  The first two families still give $O_a(n^2)$.  In family~\textup{(E3)}, $a=4$ and $n=2b+3$.  A $C_6$ meeting three $K_2$ components uses all three vertices of the fixed $K_3$ and one vertex from each chosen component.  For each triple of components there are at most $2^3\cdot6=48$ such cycles.  All other cycles meet at most two components.  Hence
\[
N(C_6,K_3\vee bK_2)\le48\binom b3+O(b^2)=n^3+O(n^2),
\]
whereas $E_{4,3}^{(\kappa)}(n)=4n^3+O(n^2)$.

Finally, let $G\subseteq M_{r,t}=K_2\vee(S_r\cup tK_2)$ in family~\textup{(E4)}.  Dividing the cycles according to their use of the $t$ independent edges gives
\[
N(C_6,M_{r,t})=r(r-1)(r-2)+2t(r^2+r+t-1).
\]
Since $|V(M_{r,t})|=r+2t+3$, the number of $C_6$ in $H(|V(M_{r,t})|,6)$ is $\fall{r+2t}3$.  Subtracting the last display gives
\[
2t(2r^2+6rt-7r+4t^2-7t+3)\ge0.
\]
For $t\ge1$, the expression in parentheses is increasing in $r\ge1$ except possibly when $t=1$, where it equals $2r^2-r$.  The cases $r=0$ are immediate.  Thus $N(C_6,G)\le E_{3,3}^{(0)}(n)\le E_{3,3}^{(\kappa)}(n)$.
\end{proof}

\begin{lemma}\label{lem:even-model-reduction}
Let $a\ge3$ and let $G$ be a sufficiently large subgraph of $H(n,2a+1)$ with $\delta(G)\ge a$ and $\circum(G)\le2a$.  Then $G\subseteq H(n,2a)$.
\end{lemma}

\begin{proof}
Let $A$ be the clique side and let $xy$ be the only possible edge in the other side.  If $xy\notin E(G)$, the conclusion follows.  Suppose $xy\in E(G)$.  Each of $x,y$ has at least $a-1$ neighbors in $A$, and every other vertex outside $A$ is adjacent to all vertices of $A$, since its degree is at least $a$.  Choose distinct $c_1\in N_G(y)\cap A$ and $c_a\in N_G(x)\cap A$, order the other clique vertices as $c_2,\ldots,c_{a-1}$, and choose distinct vertices $z_1,\ldots,z_{a-1}$ outside $A\cup\{x,y\}$.  Then
\[
xyc_1z_1c_2z_2\cdots c_{a-1}z_{a-1}c_ax
\]
is a cycle of length $2a+1$, a contradiction.
\end{proof}

\begin{lemma}\label{lem:odd-terminal}
Fix integers $a>s\ge3$ and $\kappa\in\{0,1\}$.  For all sufficiently large $n$, every $2$-connected $n$-vertex graph $G$ satisfying
\[
\circum(G)\le2a+\kappa,
\qquad \delta(G)\ge a
\]
satisfies
\[
N(C_{2s+1},G)\le F_{a,s}^{(\kappa)}(n).
\]
\end{lemma}

\begin{proof}
For large $n$, Dirac's theorem~\cite{Dirac1952} gives
\[
\circum(G)\ge\min\{n,2\delta(G)\}.
\]
Thus $\delta(G)\ge a+1$ would imply $\circum(G)\ge2a+2$, a contradiction.  Hence $\delta(G)=a$.  Since $n>2a+1\ge\circum(G)$, \cref{thm:stability} applies with $k=a$.

The alternative $\circum(G)\ge2a+2$ is excluded.  If $\kappa=0$, \cref{lem:even-model-reduction} reduces the two model alternatives to $H(n,2a)$.  If $\kappa=1$, both model graphs are subgraphs of $H(n,2a+1)$.  All remaining alternatives are handled by \cref{lem:exceptional-count}.
\end{proof}

\begin{lemma}\label{lem:even-terminal}
Fix $a\ge s\ge3$ and $\kappa\in\{0,1\}$.  For all sufficiently large $n$, every $2$-connected $n$-vertex graph $G$ with $\circum(G)\le2a+\kappa$ and $\delta(G)\ge a$ satisfies
\[
N(C_{2s},G)\le E_{a,s}^{(\kappa)}(n).
\]
\end{lemma}

\begin{proof}
Dirac's theorem~\cite{Dirac1952} gives $\circum(G)\ge\min\{n,2\delta(G)\}$.  Hence $\delta(G)\ge a+1$ would give a cycle of length at least $2a+2$, so $\delta(G)=a$.  Apply \cref{thm:stability}.  The long-cycle alternative is excluded.  If $\kappa=0$, \cref{lem:even-model-reduction} reduces the model alternatives to $H(n,2a)$.  If $\kappa=1$, both are subgraphs of $H(n,2a+1)$.  The exceptional alternatives are bounded by \cref{lem:even-exceptional-count}.
\end{proof}

\subsection{Peeling and completion of the extremal theorem}\label{sec:peeling}

The preceding lemmas give the required bound for large two-connected graphs of minimum degree at least $a$.  We now extend this bound to arbitrary graphs.  \Cref{lem:telescoping} bounds the total loss along a sequence of deletion steps, and \cref{lem:terminal-structure} shows that repeated vertex and terminal-block deletions end with an empty or two-connected graph.  \Cref{lem:two-connected-framework} combines these two lemmas with the low-degree deletion inequality and the high-minimum-degree estimate.  Finally, \cref{lem:components,lem:blocks-concentration,lem:general-framework} combine components and blocks and pass from two-connected graphs to arbitrary graphs.

For the peeling arguments, extend each function $F_{a,s}^{(\kappa)}$ and $E_{a,s}^{(\kappa)}$ to the finitely many integers below the natural range of the construction, in any nonnegative way with
\[
F_{a,s}^{(\kappa)}(0)=F_{a,s}^{(\kappa)}(1)
=E_{a,s}^{(\kappa)}(0)=E_{a,s}^{(\kappa)}(1)=0.
\]
This does not affect any sufficiently large value or any asymptotic formula.

For an integer $d\ge2$, call a function $\Phi:\N\to\R_{\ge0}$ \emph{admissible of degree $d$} if $\Phi(0)=\Phi(1)=0$ and, for some constant $c>0$,
\[
\Phi(n)=cn^d+O(n^{d-1}),
\qquad
\Phi(n)-\Phi(n-1)=cdn^{d-1}+O(n^{d-2}).
\]

\begin{lemma}\label{lem:telescoping}
Let $\Phi$ be admissible of degree $d$.  Fix $D\ge2$ and $K\ge0$.  There exists $N_{\mathrm{peel}}$ such that the following holds.  Let
\[
q_0>q_1>\cdots>q_t\ge0,
\qquad q_i\ge N_{\mathrm{peel}}\quad(i<t),
\]
and let $A_i\ge0$.  Suppose that every step is of one of the following forms.
\begin{enumerate}[label=(\roman*),leftmargin=2.4em]
\item $q_{i+1}=q_i-1$ and
\[
A_i\le\Phi(q_i)-\Phi(q_i-1).
\]
\item $q_{i+1}=q_i-r_i$, where $1\le r_i\le D$, and $A_i\le K$.
\item There are $b_i>D$ and $\varepsilon_i\in\{0,1\}$ such that
\[
q_{i+1}=q_i-b_i+\varepsilon_i,
\qquad A_i\le\Phi(b_i).
\]
When $\varepsilon_i=1$, assume $q_{i+1}\ge1$.
\end{enumerate}
Then
\[
\sum_{i=0}^{t-1}A_i+\Phi(q_t)\le\Phi(q_0).
\]
If, in addition, for some $R\ge0$,
\[
A_0+\Phi(q_1)+R\le\Phi(q_0),
\]
then
\[
\sum_{i=0}^{t-1}A_i+\Phi(q_t)+R\le\Phi(q_0).
\]
\end{lemma}

\begin{proof}
Let $c>0$ be the leading coefficient of $\Phi$.
For fixed $1\le r\le D$,
\[
\Phi(q)-\Phi(q-r)=rcd q^{d-1}+O_{\Phi,D}(q^{d-2}),
\]
so $\Phi(q)-\Phi(q-r)\ge K$ for all sufficiently large $q$, uniformly in $r$.  We also have
\begin{equation}\label{eq:merge-potential}
\Phi(x)+\Phi(y)\le\Phi(x+y-\varepsilon)
\end{equation}
when $x>D$, $\varepsilon\in\{0,1\}$, $y\ge\varepsilon$, and $x+y-\varepsilon$ is sufficiently large.  The cases $(y,\varepsilon)=(0,0)$ and $(1,1)$ are equalities.  If one variable is bounded and the other tends to infinity, the discrete derivative shows that the increase on the right has order $n^{d-1}$ and dominates the bounded term.  If both variables tend to infinity, the positive quantity $(x+y-\varepsilon)^d-x^d-y^d$ dominates the $O((x+y)^{d-1})$ errors.  This proves \eqref{eq:merge-potential}.

Choose $N_{\mathrm{peel}}$ so that these two estimates hold.  A step of type~(i) satisfies $A_i+\Phi(q_{i+1})\le\Phi(q_i)$ by assumption.  For type~(ii), this follows from the bounded-difference estimate.  For type~(iii), apply \eqref{eq:merge-potential} with $x=b_i$, $y=q_{i+1}$, and $\varepsilon=\varepsilon_i$.  Thus every step satisfies this one-step inequality, and summing it proves the first assertion.  For the reserve form, use the stronger inequality at the first step and the one-step inequalities afterward.
\end{proof}

\begin{lemma}\label{lem:terminal-structure}
For every fixed $a\ge2$ and $\kappa\in\{0,1\}$ there is $D_0=D_0(a)$ with the following property.  Start from a $2$-connected graph $G_0$ with $\circum(G_0)\le2a+\kappa$.  Repeatedly perform one of the following operations on the current induced graph $H$.
\begin{enumerate}[label=(\alph*),leftmargin=2.4em]
\item Delete a vertex of degree at most $a-1$ in $H$.
\item If no such vertex exists, choose a terminal block $B$ with $|B|\le D_0$.  If $B$ contains a cutvertex $c(B)$ of its component, delete $V(B)\setminus\{c(B)\}$.  Otherwise delete all of $V(B)$.
\end{enumerate}
When neither operation is possible, the remaining graph is empty or $2$-connected.
\end{lemma}

\begin{proof}
Choose $D_0\ge5$.  Let $\widehat G$ be a nonempty graph at which neither operation is possible, and suppose for a contradiction that $\widehat G$ is not $2$-connected.  Since operation~(a) is unavailable, $\delta(\widehat G)\ge a$.
Every graph that is disconnected or connected but not $2$-connected has at least two terminal blocks.  Indeed, if it has at least two components, choose one terminal block in each of two components.  If it is connected and not $2$-connected, the block tree has at least two leaves.  Let $B_1,B_2$ be two such terminal blocks.

Neither $B_i$ is a bridge or an isolated vertex.  Otherwise $B_i$ would contain a vertex that is not a cutvertex of its component and has degree at most one in $\widehat G$, contradicting $\delta(\widehat G)\ge a$.  Hence $B_1$ and $B_2$ are $2$-connected.  Operation~(b) is unavailable, so $|B_i|>D_0\ge5$ for $i=1,2$.
Distinct blocks are either disjoint or meet in one cutvertex.  We treat the two possibilities separately.

Suppose first that $V(B_1)\cap V(B_2)=\varnothing$.  Apply \cref{lem:two-connections}(i) in the original graph $G_0$ to the sets $V(B_1)$ and $V(B_2)$.  We obtain two vertex-disjoint paths $P_1,P_2$ joining $B_1$ to $B_2$.  Trim each path at its first and last vertices in $B_1\cup B_2$.  The two endpoints of $P_1,P_2$ in $B_i$ are distinct.  Denote them by $x_{1i},x_{2i}$.

Every vertex of $B_i$ that is not a cutvertex of its component has all its neighbours in $B_i$, so it has degree at least $a$ in $B_i$.  The terminal block $B_i$ contains at most one cutvertex.  After the two roots $x_{1i},x_{2i}$ have been excluded, at most one further vertex can fail to have degree at least $a$ in $B_i$.  Therefore at least
\[
|B_i|-3\ge\frac{|B_i|-1}{2}
\]
vertices of $B_i-\{x_{1i},x_{2i}\}$ have degree at least $a$ in $B_i$.  By \cref{lem:LiNing}, there is an $x_{1i}$--$x_{2i}$ path $Q_i\subseteq B_i$ of length at least $a$.  The paths $Q_1,Q_2,P_1,P_2$ are internally disjoint in the required way and form a cycle in $G_0$ of length at least $a+a+1+1=2a+2$,
contrary to $\circum(G_0)\le2a+\kappa$.

It remains to suppose that $V(B_1)\cap V(B_2)=\{u\}$.
Since $G_0$ is $2$-connected, $G_0-u$ is connected.  Hence it contains a path $P$ from a vertex
\[
x_1\in V(B_1)\setminus\{u\}
\]
to a vertex
\[
x_2\in V(B_2)\setminus\{u\}.
\]
Choose $P$ with its first and last contacts with $B_1\cup B_2$ as endpoints.  Its internal vertices then avoid $B_1\cup B_2$.  The path $P$ has length at least two.  Indeed, if $x_1x_2$ were an edge, then this edge together with an $x_1$--$u$ path in $B_1$ and a $u$--$x_2$ path in $B_2$ would lie on a cycle meeting both blocks, contradicting the maximality of the two distinct blocks.

As above, all but at most one vertex of $B_i$ outside the two prescribed roots $x_i,u$ have degree at least $a$ in $B_i$.  Thus \cref{lem:LiNing} gives an $x_i$--$u$ path $Q_i\subseteq B_i$ of length at least $a$.  The union
\[
Q_1\cup P\cup Q_2
\]
is a cycle of length at least
\[
a+2+a=2a+2,
\]
again a contradiction.  Therefore the terminal graph is empty or $2$-connected.
\end{proof}

\begin{lemma}\label{lem:two-connected-framework}
Fix integers $a\ge3$, $q\ge6$, $d\ge2$, and $\kappa\in\{0,1\}$.  Let $\Phi$ be admissible of degree $d$.
Assume that, for all sufficiently large graphs $G$ with $\circum(G)\le2a+\kappa$, the following hold:
\begin{enumerate}[label=(\roman*),leftmargin=2.4em]
\item if $v\in V(G)$ has degree at most $a-1$, then
\[
N(C_q,G)-N(C_q,G-v)\le\Phi(|V(G)|)-\Phi(|V(G)|-1)-\gamma |V(G)|^{d-1}
\]
for a fixed $\gamma>0$.
\item if $G$ is $2$-connected and $\delta(G)\ge a$, then $N(C_q,G)\le\Phi(|V(G)|)$.
\end{enumerate}
Then every sufficiently large $2$-connected graph $G$ with $\circum(G)\le2a+\kappa$ satisfies $N(C_q,G)\le\Phi(|V(G)|)$.
\end{lemma}

\begin{proof}
Let $D_0$ be given by \cref{lem:terminal-structure}, put $K_0=D_0^q$, and choose a common threshold $Q$ for the two assumptions and \cref{lem:telescoping}.  Increase $Q$ so that it exceeds $D_0$.  Since the increments of $\Phi$ have positive leading term, the reserve in~(i) and the bounded size of a terminal block allow us to choose $n_0\ge Q$ such that the first deletion from a graph of order at least $n_0$ satisfies
\[
A_0+\Phi(q_1)+Q^q\le\Phi(q_0),
\]
where $q_i$ is the current order and $A_i$ is the number of destroyed copies of $C_q$.

Apply the operations in \cref{lem:terminal-structure} while the current order is at least $Q$.  A vertex deletion satisfies the type~(i) inequality in \cref{lem:telescoping}.  If a terminal block $B$ is deleted, at most $D_0$ vertices are removed and every destroyed cycle lies in $B$, so $A_i\le |V(B)|^q\le K_0$.  Hence this is a type~(ii) step.  If at least one deletion is made, the reserve form of \cref{lem:telescoping} gives
\begin{equation}\label{eq:peeling-reserve-final}
\sum_{i=0}^{t-1}A_i+\Phi(q_t)+Q^q\le\Phi(q_0).
\end{equation}

If $q_t<Q$, then $N(C_q,G_t)<Q^q$, and \eqref{eq:peeling-reserve-final} proves the result.  Otherwise no operation is available.  By \cref{lem:terminal-structure}, $G_t$ is empty or $2$-connected, and in the latter case it has minimum degree at least $a$.  Assumption~(ii) gives $N(C_q,G_t)\le\Phi(q_t)$, and the same inequality is trivial when $G_t$ is empty.  This and \eqref{eq:peeling-reserve-final} finish the proof.  If no deletion is made, assumption~(ii) applies to the initial graph directly.
\end{proof}

\begin{lemma}\label{lem:components}
Let $\Phi$ be admissible of degree $d$.  For every fixed $N,K>0$ there is $n_0$ such that, whenever $n_1+\cdots+n_t=n$,
\[
\sum_{i:n_i\ge N}\Phi(n_i)+K\sum_{i:n_i<N}n_i\le\Phi(n)
\]
for all $n\ge n_0$.
\end{lemma}

\begin{proof}
Let $c>0$ be the leading coefficient of $\Phi$.
Choose $N_*\ge N$ so large that $\Phi(x)+\Phi(y)\le\Phi(x+y)$ for all $x,y\ge N_*$.  This follows because the positive leading term $(x+y)^d-x^d-y^d$ dominates the errors.  Put
\[
K_*=\max\left\{K,\max_{N\le m<N_*}\frac{\Phi(m)}m\right\}.
\]
Repeatedly merge all parts of order at least $N_*$ and charge every other part linearly.  The left side is at most $\Phi(M)+K_*R$, where $M$ is the total order of the large parts, $R=n-M$, and the term $\Phi(M)$ is omitted if there is no large part.  If $M=0$, then $K_*n\le\Phi(n)$ for large $n$.  If $M>0$ and $R$ is bounded, the discrete derivative gives $\Phi(M+R)-\Phi(M)\ge K_*R$.  If $R\to\infty$ and $R=o(M)$, this difference is $(cd+o(1))RM^{d-1}$.  If $R$ is comparable with or larger than $M$, it has order $(M+R)^d$.  In every case it dominates $K_*R$ for sufficiently large $n$.
\end{proof}

\begin{lemma}\label{lem:blocks-concentration}
Let $\Phi$ be admissible of degree $d$.  For every fixed $N,K>0$ there is $m_0$ such that the following holds.  If
\[
b_1,\ldots,b_t\ge2,
\qquad
\sum_{i=1}^t(b_i-1)=m-1,
\]
then, for every $m\ge m_0$,
\begin{equation}\label{eq:block-concentration}
\sum_{i:b_i\ge N}\Phi(b_i)
+K\sum_{i:b_i<N}(b_i-1)
\le\Phi(m).
\end{equation}
\end{lemma}

\begin{proof}
Let $c>0$ be the leading coefficient of $\Phi$.
Choose $N_*\ge N$ so large that
\begin{equation}\label{eq:block-merge}
\Phi(x)+\Phi(y)\le\Phi(x+y-1)
\end{equation}
for all $x,y\ge N_*$.  This follows from the positive leading coefficient.  If $x,y\ge N_*$, then
\[
(x+y-1)^d-x^d-y^d
\]
is positive, and by choosing $N_*$ sufficiently large its main term dominates the $O((x+y)^{d-1})$ error.

Put
\[
K_*=
\max\left\{
K,
\max_{N\le b<N_*}\frac{\Phi(b)}{b-1}
\right\}.
\]
All indices with $b_i<N_*$ may be charged by $K_*(b_i-1)$.  Repeated application of \eqref{eq:block-merge} to the remaining indices gives
\[
\sum_{i:b_i\ge N}\Phi(b_i)
+K\sum_{i:b_i<N}(b_i-1)
\le\Phi(M)+K_*R,
\]
where
\[
M=1+\sum_{i:b_i\ge N_*}(b_i-1),
\qquad
R=\sum_{i:b_i<N_*}(b_i-1),
\qquad
M+R=m.
\]
If there is no index with $b_i\ge N_*$, the right-hand side is $K_*(m-1)\le\Phi(m)$ for large $m$.  Otherwise $M\ge N_*$.  If $R$ is bounded, the discrete derivative of $\Phi$ gives
\[
\Phi(M+R)-\Phi(M)\ge K_*R
\]
for large $M$.  If $R\to\infty$ and $R=o(M)$, this difference is $(cd+o(1))RM^{d-1}$.  If $R$ is comparable with or larger than $M$, it is $\Omega((M+R)^d)$.  In every case it dominates $K_*R$ for sufficiently large $m$.  This proves \eqref{eq:block-concentration}.
\end{proof}

\begin{lemma}\label{lem:general-framework}
Fix $q\ge6$ and let $\Phi(n)=N(C_q,H(n,L))$.  Suppose that $\Phi$ is admissible of degree $d$.  If every sufficiently large $2$-connected graph $B$ with $\circum(B)\le L$ satisfies $N(C_q,B)\le\Phi(|V(B)|)$, then every sufficiently large graph $G$ with $\circum(G)\le L$ satisfies $N(C_q,G)\le\Phi(|V(G)|)$.
\end{lemma}

\begin{proof}
First let $C$ be a connected graph of order $m$, and let $B_1,\ldots,B_t$ be its blocks of order at least two, with $b_i=|V(B_i)|$.  Every cycle lies in a unique block and the block-tree identity gives $\sum_i(b_i-1)=m-1$.  Choose $N$ above the threshold in the hypothesis.  For $b_i\ge N$, we have $N(C_q,B_i)\le\Phi(b_i)$.  For $b_i<N$, choose a constant $K$ such that $N(C_q,B_i)\le K(b_i-1)$.  Hence
\[
N(C_q,C)\le\sum_{i:b_i\ge N}\Phi(b_i)+K\sum_{i:b_i<N}(b_i-1)\le\Phi(m)
\]
for all sufficiently large $m$, by \cref{lem:blocks-concentration}.

Now let $G_1,\ldots,G_t$ be the components of an arbitrary graph $G$, with orders $n_1,\ldots,n_t$.  Components above the preceding threshold contribute at most $\Phi(n_i)$, while smaller components contribute at most $K_1n_i$ for a fixed $K_1$.  Applying \cref{lem:components} gives $N(C_q,G)\le\Phi(|V(G)|)$ for all sufficiently large $|V(G)|$.
\end{proof}

\begin{proof}[Proof of \cref{thm:main}]
The cases $q=4,5$ are due to Zhu, Gy\H{o}ri, He, Lv, Salia and Xiao~\cite{ZhuEtAl2023}.  Let $q\ge6$ and write $L=2a+\kappa$, where $\kappa\in\{0,1\}$.

By \cref{lem:candidate-count,lem:even-candidate-count} and the extensions made above, $F_{a,s}^{(\kappa)}$ and $E_{a,s}^{(\kappa)}$ are admissible of degree $s$ whenever their parameters are in the stated ranges.

Suppose first that $q=2s$.  Then $a\ge s$, and put $\Phi=E_{a,s}^{(\kappa)}$.  The low-degree hypothesis of \cref{lem:two-connected-framework} is \cref{lem:even-low-degree-deletion}.  Its high-minimum-degree hypothesis is \cref{lem:even-terminal}.  Therefore \cref{lem:two-connected-framework} gives $N(C_{2s},G)\le\Phi(|V(G)|)$ for every sufficiently large two-connected graph $G$ with circumference at most $L$.  Applying \cref{lem:general-framework} then gives the same bound for every sufficiently large graph.  This argument also allows $L=q$.

Now let $q=2s+1$.  Since $L>q$, we have $a>s$, and put $\Phi=F_{a,s}^{(\kappa)}$.  In this case \cref{lem:odd-low-degree-deletion} and \cref{lem:odd-terminal} give the two hypotheses of \cref{lem:two-connected-framework}.  Hence the required bound first holds for every sufficiently large two-connected graph and then, by \cref{lem:general-framework}, for every sufficiently large graph.  In both parity cases, $H(n,L)$ has circumference at most $L$ and gives equality.
\end{proof}

\section{Forbidden paths}\label{sec:forbidden-paths}

We finish by converting the circumference result into the forbidden-path result.  The next lemma shows that every sufficiently large connected component of a $P_{p+1}$-free graph has circumference at most $p-1$.

\begin{lemma}\label{lem:path-to-cycle}
Let $p\ge3$ be an integer.  Let $G$ be a connected graph with no copy of $P_{p+1}$.  If $|V(G)|>p$, then $\circum(G)\le p-1$.
\end{lemma}

\begin{proof}
A cycle of length at least $p+1$ contains $P_{p+1}$.  Suppose that $G$ contains a cycle $C$ of length $p$.  Since $G$ is connected and has more than $p$ vertices, take a shortest path $x_0x_1\cdots x_t$ from a vertex outside $C$ to $C$, where $x_t\in V(C)$.  If $t\ge p$, then $x_0x_1\cdots x_p$ is already a copy of $P_{p+1}$.  Hence $1\le t\le p-1$.  By the choice of the path, $x_0,\ldots,x_{t-1}$ lie outside $C$.  Starting at $x_t$, take $p-t$ consecutive edges of $C$.  Since $p-t<p$, these cycle edges do not return to $x_t$.  Their union with $x_0x_1\cdots x_t$ is therefore a simple path with $p$ edges, a contradiction.
\end{proof}

\begin{proof}[Proof of \cref{thm:path-main}]
If $p\le q-1$, every copy of $C_q$ contains $P_{p+1}$, proving part~(i).

Let $p=q$.  A connected component containing a copy of $C_q$ has at least $q$ vertices.  If it had more than $q$ vertices, \cref{lem:path-to-cycle} would exclude that cycle.  Hence every such component has exactly $q$ vertices.  There are at most $\lfloor n/q\rfloor$ such components, and each contains at most the number of Hamilton cycles in $K_q$, namely $(q-1)!/2$.  This proves the upper bound in part~(ii), and disjoint copies of $K_q$ together with isolated vertices give equality.

For part~(iii), put $L=p-1$ and $a=\lfloor L/2\rfloor$.  The assumptions imply either $L>q$, or $q$ is even and $L=q$, so \cref{thm:main} applies.  Write $V(H(n,L))=A\sqcup I$, where $A$ is the clique.  If $L=2a$, then $I$ is independent.  Every path $Q$ therefore satisfies $|V(Q)\cap I|\le|V(Q)\cap A|+1$, and so $|V(Q)|\le2a+1=L+1=p$.  If $L=2a+1$, the graph $H(n,L)[I]$ has one edge.  Hence $|V(Q)\cap I|\le|V(Q)\cap A|+2$ and $|V(Q)|\le2a+2=L+1=p$.  Thus $H(n,L)$ is $P_{p+1}$-free and gives the lower bound.

For the upper bound, let $G_1,\ldots,G_t$ be the components of an $n$-vertex $P_{p+1}$-free graph $G$, and put $n_i=|V(G_i)|$ and $\Phi(m)=N(C_q,H(m,L))$.  Choose a constant $N>p$ above the threshold in \cref{thm:main}.  If $n_i\ge N$, then \cref{lem:path-to-cycle} gives $\circum(G_i)\le L$, and \cref{thm:main} gives $N(C_q,G_i)\le\Phi(n_i)$.  A component of order less than $N$ contains at most $K n_i$ copies of $C_q$, for some constant $K=K(q,N)$.  If $q=2s$, \cref{lem:even-candidate-count} gives the two asymptotic estimates for $\Phi$ required in \cref{lem:components}.  If $q=2s+1$, they are \eqref{eq:F-asymp} and \eqref{eq:F-diff}.  That lemma now gives $N(C_q,G)\le\Phi(n)$ for all sufficiently large $n$.
\end{proof}

\section*{Declaration of generative AI use}

During the preparation of this manuscript, the authors used OpenAI Codex to improve readability, to suggest shorter presentations of selected calculations, to assist in exploring the case $s=3$ of \cref{thm:main}, and to reorganize the proof of \cref{lem:kernel-normal-form}.  The authors independently checked and rewrote all AI-assisted material, including every statement, proof, and calculation, and take full responsibility for the mathematical content.

\end{document}